\documentclass[10pt,twoside, a4paper, english, reqno]{amsart}
\usepackage[dvips]{epsfig}
\usepackage{amscd}
\usepackage{a4wide}
\usepackage{amssymb}
\usepackage{amsthm}
\usepackage{amsmath}
\usepackage{bm} 
\usepackage{latexsym}
\usepackage{enumitem}
\usepackage{tabularx}
\numberwithin{figure}{section}
\numberwithin{table}{section}
\usepackage{graphicx,caption}
\DeclareMathOperator{\sech}{sech}
\DeclareMathOperator{\csch}{csch}
\usepackage{calc}
\usepackage{graphicx}
\usepackage{amsmath,amsthm,bm,mathrsfs}
\usepackage[utf8]{inputenc} 
\usepackage{url}            
\usepackage{hyperref}
\hypersetup{
    colorlinks=true,
    linkcolor=blue,
    filecolor=red,      
    urlcolor=blue,
    citecolor=red,
    }
\usepackage{amsfonts}       
\usepackage{nicefrac}       
\usepackage{microtype}      
\theoremstyle{plain}
\usepackage{doi}

\usepackage{upref}
\usepackage{hyperref}
\usepackage{ulem}
\usepackage{enumitem}
\newcommand{\norm}[1]{\left\Vert#1\right\Vert}

\newcommand{\R}{\mathbb R}
\newcommand{\Z}{\mathbb{Z}}
\newcommand{\N}{\mathbb N}

\newtheorem{theorem}{Theorem}[section]

\newtheorem{lemma}[theorem]{Lemma}
\newtheorem{definition}[theorem]{Definition}

\newtheorem{remark}[theorem]{Remark}

\numberwithin{equation}{section}     
\numberwithin{figure}{section}
\numberwithin{table}{section}

\allowdisplaybreaks

\newcounter{asnr}

{\ifnum\value{asnr}=0 \stepcounter{asnr} 
  \begin{enumerate}[label=\textbf{A}.\arabic{enumi}]
    \else
    \begin{enumerate}[label=\textbf{A}.\arabic{enumi},resume] \fi}
{\end{enumerate}}

\newcounter{defnr}

{\ifnum\value{defnr}=0 \stepcounter{defnr} 
  \begin{enumerate}[label=\textbf{D}.\arabic{enumi}]
    \else
    \begin{enumerate}[label=\textbf{D}.\arabic{enumi},resume] \fi}
{\end{enumerate}}

\numberwithin{equation}{section} \allowdisplaybreaks

\title[Stability and convergence for KdV]
{Convergence of a conservative Crank-Nicolson finite difference scheme for the KdV equation with smooth and non-smooth initial data}

\date{}

\author[M. Dwivedi]{Mukul Dwivedi}
\address[Mukul Dwivedi]{\newline
Department of Mathematics, 
	Indian Institute of Technology Jammu,
	Jagti, NH-44 Bypass Road, Post Office Nagrota,
	Jammu - 181221, India}
\email[]{mukul.dwivedi@iitjammu.ac.in}

\author[T. Sarkar]{Tanmay Sarkar}
\address[Tanmay Sarkar]{\newline
	Department of Mathematics, 
	Indian Institute of Technology Jammu,
	Jagti, NH-44 Bypass Road, Post Office Nagrota,
	Jammu - 181221, India}
\email[]{tanmay.sarkar@iitjammu.ac.in}



\thanks{}

\begin{document}
\begin{abstract}
In this paper, we study the stability and convergence of a fully discrete finite difference scheme for the initial value problem associated with the Korteweg-De Vries (KdV) equation. We employ the Crank-Nicolson method for temporal discretization and establish that the scheme is \(L^2\)-conservative. The convergence analysis reveals that utilizing inherent \emph{Kato's local smoothing effect}, the proposed scheme converges to a classical solution for sufficiently regular initial data $u_0 \in H^{3}(\mathbb{R})$ and to a weak solution in \(L^2(0,T;L^2_{\text{loc}}(\mathbb{R}))\) for non-smooth initial data $u_0 \in L^2(\mathbb{R})$. Optimal convergence rates in both time and space for the devised scheme are derived. The theoretical results are justified through several numerical illustrations.
\end{abstract}

\maketitle

\section{Introduction}
The Korteweg–de Vries (KdV) equation, a cornerstone in the study of nonlinear dispersive long waves, finds applications in diverse fields such as inverse scattering methods and plasma physics \cite{Bona1975KdV,holden1999operator,korteweg1895xli,sjoberg1970korteweg}. It describes the evolution of weakly nonlinear and weakly dispersive waves in one spatial dimension. 
The historical significance of KdV equation stems from its emergence in analysing surface water waves and its pivotal role in soliton theory. Notably, the KdV equation supports soliton solutions—persistent, stable solitary waves that arise from a delicate balance between nonlinearity and dispersion \cite{linares2014introduction}. These solitons play a crucial role in understanding wave interactions and propagation phenomena. Motivated by this,
we consider the following Cauchy problem associated with the KdV equation
\begin{equation}\label{kdveqn}
\begin{cases}
    u_t + uu_x + u_{xxx} = 0, & \qquad (x,t) \in Q_T :=\mathbb{R}\times (0,T), \\
    u(0,x) = u_0(x), & \qquad x \in \mathbb{R},
\end{cases}
\end{equation}
where $u : Q_T \xrightarrow[]{}\mathbb{R} $ is the unknown function, $T>0$ is fixed, and $u_0$ is the prescribed initial data.

The well-posedness of the initial value problem \eqref{kdveqn} has been the subject of extensive investigation in the literature. Bona and Smith in \cite{Bona1975KdV} made pioneering contributions by establishing the first results on local and global well-posedness for the KdV equation. Specifically, they demonstrated local well-posedness for initial data in $H^s$ with $s > 3/2$ and global well-posedness for $s \geq 2$. Building on this foundation, subsequent research by Kenig et al. \cite{kenig1993cauchy} and Killip et al. \cite{killip2019kdv} extended the result of global well-posedness to encompass initial data in negative-order Sobolev spaces.

The numerical computation of solutions for \eqref{kdveqn} presents inherent challenges. It is well-known that the equation \eqref{kdveqn} exhibits two competing effects that contribute to the difficulties encountered in the numerical approximation process.
The inclusion of the nonlinear convection term \(uu_x\) in equations like the Burgers equation \(u_t + uu_x = 0\) results in the emergence of infinite gradients in finite time even for initially smooth data \cite{klein2015numerical, tao2004global}. Additionally, the presence of the linear dispersive term \( u_{xxx}\) inherent in the  KdV equation introduces dispersive waves that are arduous to compute with high accuracy and efficiency. Consequently, due to the combined effects of the nonlinear convection term and the dispersive term, accurate and efficient numerical methods for the KdV equation remain a highly intricate task.

Sjöberg \cite{sjoberg1970korteweg} initiated the convergence analysis of the KdV equation through a semidiscrete scheme for the initial data in $H^3(\R)$. This also yields global well-posedness of \eqref{kdveqn} due to its conservation in the $L^2$ norm. Afterwards, the numerical treatment with convergence analysis of the scheme for the KdV equation has also garnered significant attention in the last few decades, leading to the development of various computational methods. For instance, Amorim and Figueira \cite{Amorim2013mKdV} introduced a semi-discrete finite difference method designed for $L^2$ initial data by introducing a fourth order stabilization term. However, the study lacked a fully discrete convergence analysis, and conclusive evidence from the numerical illustrations. Holden et al. \cite{holden2015convergence} introduced a fully discrete finite difference scheme for the KdV equation \eqref{kdveqn} applicable to both $H^3$ and $L^2$ initial data, which is shown to be first order accurate numerically. Recently, Court{\`e}s et al. \cite{courtes2020error} designed a convergent finite difference scheme considering a $4$-point $\theta$ scheme for the dispersive term and demonstrated its first order accuracy. For other finite difference related work involving \eqref{kdveqn}, one can refer to \cite{goda1975stability, 2016Ujjwalkawahara, skogestad2009boundary, feng1998finite, li2006high, abrahamsen2021solving, wang2021high} and references therein. Apart from the finite difference approaches, there are developments in the direction of Galerkin schemes for \eqref{kdveqn}.
Dutta et al. \cite{dutta2015convergence} proposed a higher-order finite element method tailored for $L^2$ initial data. Holden et al. \cite{holden1999operator, holden2011operator} devised an operator splitting method which is also first order accurate. This technique is further generalized in \cite{dutta2021operator}. Additionally, Dutta et al. \cite{dutta2016note} presented a Crank-Nicolson Galerkin scheme specifically designed for $L^2$ initial data.

In this paper, our aim is to develop a fully discrete implicit finite difference scheme of \eqref{kdveqn} which is conservative and provides higher convergence rates. In this regard, we design an efficient finite difference scheme by discretizing the time derivative using the Crank-Nicolson method. We demonstrate the following behaviour of the scheme:
 \begin{enumerate}
     \item The proposed scheme is $L^2$ conservative. However, the implicit nature of the scheme requires proving its solvability at each time step. This will be accomplished by defining a suitable iterative scheme.
     \item We prove that the difference approximations obtained by the proposed scheme converge to a classical solution of the KdV equation \eqref{kdveqn} provided the initial data is sufficiently regular. The idea of the proof differs significantly from \cite{holden2015convergence}. 
     Whenever the initial data \(u_0 \in L^2(\mathbb{R})\), motivated by the work \cite{dutta2015convergence, dwivedi2023stability, dutta2016note}, we utilize the inherent Kato's type \cite{kato1983cauchy} local smoothing effect. We establish the discrete analogue of this effect for the approximate solution obtained through our devised finite difference scheme. This ensures the efficacy of our approach in handling non-smooth initial data, providing a comprehensive framework for the stability and convergence analysis. More precisely, ensuring the compactness of the difference approximations through the Aubin-Simon compactness lemma, we show the convergence of weak solution in the space $L^2(0,T;L^2_{\text{loc}}(\mathbb{R}))$.
     \item We investigate the theoretical convergence rates under certain assumptions on the initial data. We prove that the proposed scheme is second order accurate. Furthermore, these convergence rates are justified through the numerical experiments of one soliton and two soliton cases. Since the real solutions of IVP \eqref{kdveqn} possess mainly three conserved quantities:
    \begin{align*}
        C_1(u): & = \int_{\R} u(x,t)~dx, \quad C_2(u):= \int_{\R} u^2(x,t)~dx,\\
         C_3(u):&= \int_{\R} \left((\partial_x u)^2 - \frac{u^3}{3}\right)(x,t)~dx.
    \end{align*}
We shall demonstrate that the proposed numerical scheme conserves a discrete version of these quantities.
 \end{enumerate}

In this paper, C denotes a generic constant whose value can change in each step and it is independent of both the spatial step $\Delta x$ and time step $\Delta t$.

The rest of the paper is organized as follows: In Section \ref{sec2}, we lay the foundation by establishing preliminary estimates and introducing discrete operators. In addition, we present the proposed fully discrete finite difference scheme for the Korteweg-de Vries equation \eqref{kdveqn}.
Moving on to Section \ref{sec3}, we analyze the convergence of the approximate solutions for both regular and less regular initial data. Theoretical insights into the convergence rates are explored in Section \ref{sec4}. 
Finally, in Section \ref{sec5}, we validate our theoretical findings through a series of numerical illustrations. This comprehensive organization ensures a systematic and coherent presentation of our methodology, analysis, and numerical results.

\section{Notations and Preliminary Estimates}\label{sec2}
\subsection{Discrete operators}
We use uniform discretization of space and time using the nodal points $x_j = j\Delta x,~j\in \Z$ and $t_n = n\Delta t,~n\in\N$, where $\Delta x$ and $\Delta t$ are spatial and temporal steps respectively. The difference operators for a function $v:\mathbb{R}\to\mathbb{R}$ are defined as
\begin{equation}\label{frdiff}
     D_{\pm}v(x) = \pm\frac{1}{\Delta x}\big(v(x\pm\Delta x) - v(x)\big), \qquad D = \frac{1}{2}(D_+ +D_-).
\end{equation}
We also introduce the shift operators $$S^{\pm}v(x) = v(x\pm\Delta x),$$ and the averages $\tilde{v}(x)$ and $\bar{v}(x)$ are defined by
\[\tilde{v}(x) := \frac{1}{3}\left(S^+v(x) + v(x) + S^-v(x) \right), \qquad \bar{v}(x) := \frac{1}{2}(S^+ + S^-)v(x).\]
The difference operators satisfy the following identities
\begin{equation*}
    \begin{aligned}
        D(vw) &= \bar{v}Dw + \bar{w}Dv, \\
        D_{\pm}(vw) & = S^{\pm}vD_{\pm}w + wD_{\pm}v = S^{\pm}w D_{\pm}v+vD_{\pm}w.
    \end{aligned}
\end{equation*}
For any given function $v$, we define $v_j = v(x_j)$. Moreover for $v,w \in \ell^2$, we define the usual inner product and norm as
\begin{equation}\label{normdef}
    \langle v,w\rangle = \Delta x\sum_{j\in\mathbb{Z}}v_jw_j, \qquad \norm{v}=\norm{v}_2=\langle v,v \rangle^{1/2}.
\end{equation}
It is observed that the difference operators satisfy shifting properties within the inner product
\[ \langle v , Dw\rangle = - \langle Dv , w\rangle, \qquad \langle v , D_{\pm}w\rangle  = -\langle D_{\mp}v , w\rangle. \]
We define the discrete Sobolev $h^3$-norm of a grid function $v$ as
\begin{equation}\label{Disc_sobo}
    \norm{v}_{h^3} := \norm{v} + \norm{D_+v} + \norm{D_+D_-v} + \norm{D_-DD_+v}.
\end{equation}
As a consequence, we have 
$$\norm{v}_\infty \leq \frac{1}{\Delta x^{1/2}}\norm{v},\qquad \norm{Dv} \leq \frac{1}{\Delta x}\norm{v}. $$
Using the properties of the difference operator, we deduce the following identities
\begin{align}
    \label{Dz_1z_2}\langle D(vw) ,w\rangle &= \frac{\Delta x}{2}\langle D_+vDw ,w \rangle + \frac{1}{2}\langle S^-wDv ,w\rangle,\\
    \label{D3z_1z_2}D_-DD_+ (vw) &= D_-vD_+w + S^-vD_-DD_+w+ D_+vD_+w + D_-DD_+vDw.
\end{align}
The difference operator in time for $v:[0,T]\to\mathbb{R}$ is given by 
\[ \Delta t D_{\pm}v(t)= \pm\left(v(t+\Delta t) - v(t)\right), \quad t\in[0,T-\Delta t].\]
A fully discrete grid function $v_{\Delta x} :\Delta x\Z \times\Delta t\mathbb{N}_0\to\mathbb{R}$ is defined as 
\begin{equation*}
    v_{\Delta x}(x_j, t_n) = v_j^n, \quad j \in \mathbb{Z}, \quad n \in \mathbb{N}_0.
\end{equation*}
We omitted $\Delta t$ in the definition of grid function due to the CFL-condition. Furthermore, we denote  $v^n := \{v^n_j\}_{j\in\mathbb{Z}}$.
\subsection{Preliminary estimates}
We begin with a pivotal lemma establishing a relationship between continuous and discrete Sobolev norms. More precisely, we have the following lemma:
\begin{lemma}
    Let $u\in H^3(\mathbb{R})$.  Assume that $u_{\Delta x} = \{u(x_j)\}_{j\in\mathbb{Z}}$. Then there exists a constant $C$ such that
    \begin{equation}\label{esti:h3_main}
        \norm{u_{\Delta x}}_{h^3} \leq C\norm{u}_{H^3},
    \end{equation}
    where the norm $\|\cdot\|_{h^3}$ is defined by \eqref{Disc_sobo}.
\end{lemma}
\begin{proof}
   We observe that
    \begin{align*}
        \norm{D_-DD_+ u}^2  &= \Delta x \sum_j \left( \frac{1}{\Delta x}\left(D_-Du(x_{j+1}) - D_-Du(x_j)\right)\right)^2\\
        &= \Delta x \sum_j \left( \int_{x_j}^{x_{j+1}}\frac{1}{\Delta x}\partial_xD_-Du(x)\,dx\right)^2\\
        &\leq \Delta x \sum_j \left( \norm{\frac{1}{\Delta x}}_{L^2([x_j,x_{j+1}])}\norm{\partial_xD_-Du(x)}_{L^2([x_j,x_{j+1}])}\right)^2,
    \end{align*}    
where we have used the H{\"o}lder's inequality. Using the fact that the difference operator commutes with the continuous operator $\partial_x$, we get
    \begin{align*}
       \norm{D_-DD_+ u}^2 \leq  \norm{D_-D\partial_x u}^2.
    \end{align*}
Similar calculations yield
    \[\norm{D_-D\partial_x u}^2\leq\norm{D\partial^2_x u}^2\leq \norm{\partial^3_x u}^2\leq \norm{u}_{H^3}^2.\]
Following the similar arguments, we have 
\begin{align*}
    \|D_{+}u\| \leq \|\partial_x u\|_{L^2}, \quad \|D_{+}D_{-}u\|\leq \|\partial_x^2 u\|_{L^2}.
\end{align*}
Hence the estimate \eqref{esti:h3_main} is obtained.
\end{proof}
\subsection{Numerical scheme}
We propose the following Crank-Nicolson(CN) fully discrete finite difference scheme to obtain approximate solutions of the KdV equation \eqref{kdveqn}:
\begin{equation}\label{CNFDscheme}
    u_j^{n+1} = u_j^n -\Delta t \Tilde{u}_j^{n+1/2} Du_j^{n+1/2} - \Delta tD_-DD_+u^{n+1/2}_j,\qquad n\in\mathbb{N}_0,~j\in\mathbb{Z}.
\end{equation}
For the initial data, we have 
\begin{equation*}
    u_j^0 = u_0(x_j), \qquad j\in \mathbb{Z}.
\end{equation*}

The scheme immediately reveals its $L^2$-conservative nature. For the time being, let us assume that the scheme has a unique solution, although we will prove this in a subsequent section. Delving into the details, we perform an inner product of \eqref{CNFDscheme} with $u^{n+1/2}$ to obtain
\begin{equation}\label{temp:conserve_1}
    \|u^{n+1}\|^2 = \|u^n\|^2 - \Delta t \langle \mathbb{G}(u^{n+1/2}) , u^{n+1/2}\rangle - \Delta t \langle D_-DD_+(u^{n+1/2}), u^{n+1/2}\rangle,
\end{equation}
where $\mathbb{G}(u) = \Tilde{u}Du$. We observe that
\begin{align*}
    \langle \mathbb{G}(u) , u \rangle = 0, \qquad \langle D_-DD_+(u), u \rangle = 0.
\end{align*}
Subsequently, from \eqref{temp:conserve_1} we end up with $\|u^{n+1}\|= \|u^n\|$ for all $n\in\N$.
Finally, we interpolate the approximation $u^n$ through two steps.\\
\emph{Interpolation in space and time:}
we employ the piece-wise quadratic continuous interpolation to interpolate in space for $j \in \mathbb{Z}$,
\begin{align*}
    u^n(x) = & u^n_j + (x - x_j)D_+u^n_j + \frac{1}{2}(x - x_j)^2D_+D_-u^n_j
    + \frac{1}{6}(x - x_j)^3D_+DD_-u^n_j, \quad x \in [x_j, x_{j+1}).
\end{align*}
Afterwards, we perform the interpolation in time 
\begin{equation}\label{interCN}
    u_{\Delta x}(x, t) = u^n(x) + (t - t_n)D_+^tu^n(x), \quad x \in \mathbb{R}, \quad t \in [t_n, t_{n+1}), \quad (n+1)t_{n+1} \leq T.
\end{equation}
The regularity of the initial data $u_0$ plays a crucial role in the study of convergence of the finite difference scheme \eqref{CNFDscheme}. We focus on the convergence of the scheme \eqref{CNFDscheme} in the subsequent sections. 
\section{Convergence analysis of the scheme}\label{sec3}
Hereby our aim is to prove the convergence of approximate solutions obtained by the CN scheme \eqref{CNFDscheme} to the classical solution of \eqref{kdveqn} if the initial data $u_0\in H^3(\R)$ and to the weak solution in $L^2([0,T);L^2_\text{loc}(\R))$ if the initial data $u_0\in L^2(\R)$. Please note that the scheme is implicit in nonlinear term, hence we have to ensure that there exists a unique solution. To solve \eqref{CNFDscheme}, we will consider the fixed-point iteration technique as in \cite{dutta2016convergence, dwivedi2023stability, thomee1998numerical} and prove the solvability at each time step in the following Lemma \ref{lemma:solv.lemma} and Lemma \ref{lemma:solv.lemmaL2} for the initial data in $H^3$ and $L^2$ respectively.

We commence by introducing the sequence $\{w^\ell\}_{\ell\geq0}$, defined as the solution of the following linear iterative equation:
\begin{equation}\label{iterr}
    \begin{cases}
        w^{\ell+1} = u^n - \Delta t \mathbb{G}\left(\frac{u^n+w^\ell}{2}\right) - \Delta t D_-DD_+\left(\frac{u^n+w^{\ell+1}}{2}\right),\\
        w^0 = u^n.
    \end{cases}
\end{equation}
The linearity of the iteration in $w^{\ell+1}$ allows us to express it in a more structured form:
\begin{equation}\label{iterr2}
    \left(1+\frac{\Delta t}{2}D_-DD_+\right)w^{\ell+1} = u^n - \Delta t \mathbb{G}\left(\frac{u^n+w^\ell}{2}\right) - \frac{\Delta t}{2} D_-DD_+u^n.
\end{equation}
Since the matrix obtained by applying the operator $\frac{\Delta t}{2}D_-DD_+$ on the vector \(w^{\ell+1}\) is skew-symmetric, the resulting coefficient matrix on the left-hand side of \eqref{iterr2} is positive definite. This is evident from the skew-symmetric property of the discrete operator $D_-DD_+$,
\[ D_-DD_+ w^{\ell+1}_j = \frac{1}{2\Delta x^3}\left(w^{\ell+1}_{j+2}-w^{\ell+1}_{j+1}+w^{\ell+1}_{j-1} -w^{\ell+1}_{j-2}\right). \]
We remark that the positive definiteness of the matrix ensures the existence and uniqueness of the iterative scheme \eqref{iterr}.
\subsection{Convergence analysis with $H^3$ initial data}
Consider the case where the initial data $u_0$ is sufficiently regular. In this subsection, we establish the stability of the scheme and demonstrate that the approximate solution obtained by \eqref{CNFDscheme} converges to the classical solution of \eqref{kdveqn}. The iterative scheme \eqref{iterr} is instrumental in handling the non-linearity with an implicit term. We begin by proving a lemma that ensures the solvability of the scheme at each time step whenever the initial data $u_0$ belongs to $H^3(\R)$. The following lemma sets the foundation for the subsequent stability and convergence analysis.
\begin{lemma}\label{lemma:solv.lemma}
    Let $ K = \frac{4-L}{1-L} > 4$ be a constant with $0<L<1$.
    Consider the iteration given by \eqref{iterr} and assume that the CFL condition satisfies
    \begin{equation}\label{CNCFL}
        \lambda \leq \frac{L}{K\|u^n\|_{h^3}},
    \end{equation}
    where $\lambda = \frac{\Delta t}{\Delta x}$. Then there exists a function $u^{n+1}$ that solves \eqref{CNFDscheme} and $\lim_{\ell\xrightarrow{}\infty}w^{\ell} = u^{n+1} $. Moreover, the following estimate holds:
    \begin{equation}\label{stabn}
        \|u^{n+1}\|_{h^3} \leq K \|u^n\|_{h^3}.
    \end{equation}
\end{lemma}
\begin{proof}
Let us define $\Delta w^\ell := w^{\ell+1} - w^\ell $. Then the scheme \eqref{iterr2} can be represented as 
  \begin{equation}\label{iterrwl}
      \left(1+\frac{1}{2}\Delta t D_-DD_+\right)\Delta w^\ell =-\Delta t \Delta \mathbb{G},
  \end{equation}
where $\Delta\mathbb{G}$ is given by
\begin{align*}
    \Delta \mathbb{G} = \left[\mathbb{G}\left(\frac{u^n+w^\ell}{2}\right) - \mathbb{G}\left(\frac{u^n+w^{\ell-1}}{2}\right)\right].
\end{align*}
  Applying the discrete operator $D_-DD_+$ to \eqref{iterrwl} and taking the inner product with $D_-DD_+\Delta w^\ell$, we get
  \begin{equation*}
     \begin{split}
      \|D_-DD_+\Delta w^\ell\|^2 & = \Delta t \left\langle D_-DD_+\Delta \mathbb{G}, D_-DD_+\Delta w^\ell\right\rangle\\
       &\leq  \Delta t  \|D_-DD_+\Delta \mathbb{G}\| \|D_-DD_+\Delta w^\ell\|.
     \end{split}
  \end{equation*}
  Observe that $\Delta\mathbb{G}$ can also be represented as
  \begin{equation*}
      \Delta \mathbb{G} = - \frac{1}{4}\left[ \widetilde{\Delta w^{\ell-1}} D(u^n+w^{\ell-1}) + \widetilde{(u^n+w^\ell)}D(\Delta w^{\ell-1})\right].
  \end{equation*}
The term $\|D_-DD_+\Delta \mathbb{G}\|$ can be estimated by applying the Lemma A.1 in \cite{holden2015convergence}, so-called discrete Sobolev inequalities, $\norm{u}_{\infty}\leq \norm{u}_{h^1}$ and along with identity \eqref{D3z_1z_2}, we have
  \begin{equation*}
      \begin{split}
      \|D_-DD_+ &\big(\widetilde{\Delta w^{\ell-1}} D(u^n+w^{\ell-1})\big)\|\\
      & \leq  \|D_-\widetilde{\Delta w^{\ell-1}}D_+D(u^n+w^{\ell-1})\|+ \|\widetilde{\Delta w^{\ell-1}}D_-DD_+D(u^n+w^{\ell-1})\|\\& \quad + \|D_+\widetilde{\Delta w^{\ell-1}}D_+D(u^n+w^{\ell-1})\| \quad+ \|D_-DD_+\widetilde{\Delta w^{\ell-1}}DD(u^n+w^{\ell-1})\|
      \\& 
      \leq \|D_-\widetilde{\Delta w^{\ell-1}}\|_{\infty}\|(u^n+w^{\ell-1})\|_{h^3} + \frac{1}{\Delta x}\|\widetilde{\Delta w^{\ell-1}}\|_{\infty}\|u^n+w^{\ell-1}\|_{h^3} \\& \quad+ \|D_+\widetilde{\Delta w^{\ell-1}}\|_{\infty}\|u^n+w^{\ell-1}\|_{h^3} + \|\widetilde{\Delta w^{\ell-1}}\|_{h^3}\|DD(u^n+w^{\ell-1})\|_{\infty} \\&
      \leq \frac{1}{\Delta x} \max \left\{ \|u^n\|_{h^3} , \|w^{\ell}\|_{h^3}, \|w^{\ell-1}\|_{h^3} \right\}\|\Delta w^{\ell-1}\|_{h^3},
      \end{split}
  \end{equation*}
  and similarly, we also have the following estimate
  \begin{equation*}
      \begin{split}
      \|D_-DD_+ \big(\widetilde{(u^n+w^\ell)}D{\Delta w^{\ell-1}}\big)\| 
      \leq \frac{1}{\Delta x} \max \left\{ \|u^n\|_{h^3} , \|w^{\ell}\|_{h^3},\|w^{\ell-1}\|_{h^3} \right\}\|\Delta w^{\ell-1}\|_{h^3}.
      \end{split}
  \end{equation*}
  Combining the above estimates together, we conclude that
  \begin{equation}\label{h3boundd}
       \|D_-DD_+\Delta w^\ell\|\leq \lambda \max \left\{ \|u^n\|_{h^3} , \|w^{\ell}\|_{h^3} ,\|w^{\ell-1}\|_{h^3} \right\}\|\Delta w^{\ell-1}\|_{h^3}.
  \end{equation}
Furthermore, in a similar manner, one can estimate $\|D_+D_-\Delta w^\ell\|$, $\|D_+\Delta w^\ell\|$, and $\|\Delta w^\ell\|$ like \eqref{h3boundd}. Summing up all these estimates provides 
  \begin{equation}\label{h1bound}
       \|\Delta w^\ell\|_{h^3}\leq \lambda \max \left\{ \|u^n\|_{h^3} , \|w^{\ell}\|_{h^3} ,\|w^{\ell-1}\|_{h^3} \right\}\|\Delta w^{\ell-1}\|_{h^3}.
  \end{equation}

Our strategy for proving the convergence and boundedness of the sequence $\{w^\ell\}$ is to demonstrate that it is a Cauchy sequence. We proceed by induction, starting with the observation that $w^1$ satisfies 
\begin{equation}\label{w^1}
     w^{1} = u^n - \Delta t \mathbb{G}(u^n) - \Delta t D_-DD_+\Big(\frac{u^n+w^1}{2} \Big).
\end{equation}
We show that $w^1$ is $h^3$-bounded by applying the discrete operator $D_-DD_+$ to \eqref{w^1} and taking the inner product with $D_-DD_+(u^n+w^1)$ yields
\begin{equation*}
    \begin{split}
        \|D_-DD_+w^1\|^2 =& \|D_-DD_+u^n\|^2 -\Delta t \langle D_-DD_+\mathbb{G}(u^n) , D_-DD_+(u^n+w^1)\rangle\\
        =& \|D_-DD_+u^n\|^2 -\Delta t \langle D_-DD_+\mathbb{G}(u^n) , D_-DD_+(w^1)\rangle\\
        & - \Delta t \langle D_-DD_+\mathbb{G}(u^n) , D_-DD_+(u^n)\rangle\\
        \leq& \|D_-DD_+u^n\|^2 +\Delta t^2 \|D_-DD_+\mathbb{G}(u^n)\|^2 + \frac{1}{4}\|D_-DD_+w^1\|^2\\ 
        & +\Delta t^2 \|D_-DD_+\mathbb{G}(u^n)\|^2  + \frac{1}{4}\|D_-DD_+u^n\|^2.
    \end{split}
\end{equation*}
In order to estimate the nonlinear part, we again apply the Lemma A.1 in \cite{holden2015convergence} and along with the identity \eqref{D3z_1z_2} to get
  \begin{equation*}
  \begin{split}
       \|D_-DD_+\mathbb{G}(u^n)\| = \| D_-DD_+(\Tilde{u}^nDu^n)\| &\leq
       \|D_-\Tilde{u}^nD_+Du^n\| + \|\Tilde{u}^nD_-DD_+Du^n\| \\&\quad+ \|D_+\Tilde{u}^nD_+Du^n\|+ \|D_-DD_+\Tilde{u}^nDDu^n\| \\ &\leq
       \|D_-\Tilde{u}^n\|_{\infty}\|u^n\|_{h^3} + \frac{1}{\Delta x}\|\Tilde{u}^n\|_{\infty}\|u^n\|_{h^3}  \\&\quad+ \|D_+\Tilde{u}^n\|_{\infty}\|u^n\|_{h^3} + \|\Tilde{u}^n\|_{h^3} \|DDu^n\|_{\infty}
       \\&\leq\frac{2}{\Delta x}\|u^n\|^2_{h^3}.
  \end{split}
  \end{equation*}
Hence we end up with
  \begin{equation}\label{temh1}
      \|D_-DD_+w^1\|\leq \sqrt{\frac{4}{3}}\left(\frac{5}{4}+8\lambda^2\|u^n\|_{h^3}^2\right)^{1/2}\|u^n\|_{h^3}.
  \end{equation}
  The choice of $L$ and $K$ implies
  \begin{equation*}
       \sqrt{\frac{4}{3}}\left(\frac{5}{4}+8\lambda^2\|u^n\|_{h^3}^2\right)^{1/2}\leq 2.
  \end{equation*}
Proceeding in a similar way as above, we estimate the lower-order difference operator. Hence we have
  \begin{equation*}
      \|w^1\|_{h^3} \leq 2 \|u^n\|_{h^3}\leq K \|u^n\|_{h^3}. 
  \end{equation*}
For carrying out the induction argument, let us assume
  \begin{align}
         \label{fis} \|w^\ell\|_{h^3} &\leq K \|u^n\|_{h^3}, \quad \text{for} \quad \ell =1,...,m,\\
         \label{sec} \|\Delta w^\ell\|_{h^3} &\leq L \|\Delta w^{\ell-1}\|_{h^3}, \quad \text{for} \quad \ell  =2,...,m.
  \end{align}
We have demonstrated \eqref{fis} for $m=1$. We prove \eqref{sec} for $m=2$ using the CFL condition \eqref{CNCFL} and the estimate for $w^1$ as follows:
  \begin{equation*}
      \|\Delta w^{2}\|_{h^3} \leq \lambda\max \left\{\|u^n\|_{h^3} , \|w^1\|_{h^3} \right\} \|\Delta w^1\|_{h^3} \leq K\lambda\|u^n\|_{h^3}\|\Delta w^1\|_{h^3} \leq L \|\Delta w^1\|_{h^3}.
  \end{equation*}
Subsequently, we shall establish \eqref{fis} for $m>1$. For this, we observe that
  \begin{align*}
      \|w^{m+1}\|_{h^3} &\leq \sum_{\ell=0}^{m}\|\Delta w^\ell\|_{h^3} +\|u^n\|_{h^3} \\
      &\leq \|w^1-u^n\|_{h^3}\sum_{\ell=0}^{m} L^\ell + \|u^n\|_{h^3}\\
      &\leq \left(\|w^1\|_{h^3} +\|u^n\|_{h^3}\right) \frac{1}{1-L} + \|u^n\|_{h^3} \\
      &\leq \frac{4-L}{1-L} \|u^n\|_{h^3} = K\|u^n\|_{h^3}.
  \end{align*}
Finally, we end up with
  \begin{equation*}
      \|\Delta w^{m+1}\|_{h^3} \leq \lambda K \|u^n\|_{h^3} \|\Delta w^m\|_{h^3} \leq L\|\Delta w^m\|_{h^3},
  \end{equation*}
where \eqref{CNCFL} is incorporated.
Summing up all the above estimates, we have the desired result \eqref{stabn}, and by \eqref{sec}, it is clear that $\{w^\ell\}$ is a Cauchy sequence, hence it converges to $u^{n+1}$. This completes the proof.
\end{proof}
\begin{remark}
    The Lemma \ref{lemma:solv.lemma} implies that the solvability of the scheme \eqref{iterr} at each time step under the CFL condition \eqref{CNCFL}, where $\lambda$ depends on the $n$-th step. We need to show the ratio $\lambda$ between spatial and temporal bound  depends only on the initial data $u_0$ to demonstrate the stability of \eqref{iterr}. 
\end{remark}
The following lemma establishes local a priori bounds of the approximated solution $u^n$ under the $h^3$-norm.
\begin{lemma}\label{lemma:Stablemma}
    Let $u_0\in H^{3}(\R)$. Assume that $\Delta t$ and $\Delta x$ satisfies 
    \begin{equation}\label{CFLCNi}
        \lambda = \frac{\Delta t}{\Delta x} \leq \frac{L}{KM}
    \end{equation}
    for some $M = M(\|u_0\|_{h^3})$. Then there exist a time $T>0$, depending on $\|u_0\|_{h^3}$, such that 
    \begin{equation}\label{stabcn}
        \|u^n\|_{h^3} \leq C, \qquad \text{for  } n\Delta t\leq T,
    \end{equation}
    where the constant $C=C(\|u_0\|_{h^3})$.
\end{lemma}
\begin{proof}
Suppose that $D_-DD_+ u^n = 0$. Then $u^n = 0$ and $u^{n+1} = 0$ since $u^n,u^{n+1} \in \ell^2$. Hence \eqref{stabcn} trivially holds. Thus, we shall assume that $D_-DD_+u^n \neq 0$.
We will apply the discrete operator $D_-DD_+$ to \eqref{CNFDscheme} and then taking the inner product with $D_-DD_+ u^{n+1/2}$ yields
\begin{equation*}
   \frac{1}{2}\|D_-DD_+u^{n+1}\|^2 = \frac{1}{2}\|D_-DD_+u^n\|^2 - \Delta t \langle D_-DD_+\mathbb{G}(u^{n+1/2}), D_-DD_+u^{n+1/2}\rangle,
\end{equation*}
which further becomes
\begin{equation}
  \label{estforDt} \|D_-DD_+u^{n+1}\| - \|D_-DD_+u^n\| \leq 2\Delta t \frac{\langle D_-DD_+\mathbb{G}(u^{n+1/2}), D_-DD_+u^{n+1/2}\rangle}{\|D_-DD_+u^{n+1}\| + \|D_-DD_+u^n\|}.
\end{equation}
For writing convenience, we drop the superscript $n+1/2$ from the above expression and denote $u$ instead of $u^{n+1/2}$ for the moment. The earlier estimate \eqref{D3z_1z_2} yields
\begin{align*}
   \langle D_-DD_+\mathbb{G}(u), D_-DD_+u\rangle =& \langle D_-DD_+(\Tilde{u}Du), D_-DD_+u\rangle \\
   =& \langle D_-\Tilde{u}D_+(Du), D_-DD_+u\rangle+\langle S_-\Tilde{u}D_-DD_+(Du), D_-DD_+u\rangle\\
   &+\langle D_+\Tilde{u}D_+(Du), D_-DD_+u\rangle+\langle D_-DD_+(\Tilde{u})Du, D_-DD_+u\rangle\\
   =&:\mathcal{I}_1 + \mathcal{I}_2 + \mathcal{I}_3 + \mathcal{I}_4.
\end{align*}
Let us estimate one by one. By using the discrete Sobolev inequality $\|D_-\Tilde{u}\|_{\infty} \leq 2(\|D_-DD_+u\| + \|u\|)$, we have
\begin{align*}
   |\mathcal{I}_1| \leq  \norm{D_-\Tilde{u}}_{\infty}\norm{D_-DD_+u}^2 &\leq 2(\norm{D_-DD_+u} + \norm{u})\norm{D_-DD_+u}^2 \\
   & \leq 2\norm{D_-DD_+u} \norm{u}^2_{h^3}.
\end{align*}
In a similar way, we have the following estimates
\begin{align*}
   |\mathcal{I}_3|  \leq 2\norm{D_-DD_+u} \norm{u}^2_{h^3} \qquad \text{and} \qquad |\mathcal{I}_4|  \leq 2\norm{D_-DD_+u} \norm{u}^2_{h^3}.
\end{align*} 
We are left with the estimate of $\mathcal{I}_2$ which can be performed as follows
\begin{align*}
    \mathcal{I}_2 &= \langle S_-\Tilde{u}D_-DD_+(Du), D_-DD_+u\rangle\\
          & \stackrel{v:=D_-DD_+ u}{=} \langle S^{-}\Tilde{u}Dv, v\rangle = -\langle D(S^{-}\Tilde{u}v), v\rangle\\
          &= \frac{\Delta x}{2}\langle D_+(S^-\Tilde{u})Dv,v\rangle + \frac{1}{2}\langle S^-vD(S^-\Tilde{u}),v\rangle,
\end{align*}
and consequently, we have the following estimate
\begin{equation*}
    |\mathcal{I}_2|\leq 2\norm{D_-DD_+u} \norm{u}^2_{h^3}.
\end{equation*}
Incorporating all the estimates on $\mathcal{I}_i,~i=1,2,3,4$, we obtain
\begin{align*}
   2\frac{\left|\langle D_-DD_+\mathbb{G}(u^{n+1/2}), D_-DD_+u^{n+1/2}\rangle\right|}{\|D_-DD_+u^{n+1}\| + \|D_-DD_+u^n\|} &\leq 4\frac{\norm{D_-DD_+u^{n+1/2}} \norm{u^{n+1/2}}^2_{h^3}}{\|D_-DD_+u^{n+1}\| + \|D_-DD_+u^n\|}
   \leq 2\norm{u^{n+1/2}}^2_{h^3}.
\end{align*}
Thus the estimate \eqref{estforDt} turns into
\begin{align*}
   \norm{D_-DD_+u^{n+1}} &\leq  \|D_-DD_+u^n\| + 4\Delta t \frac{\langle D_-DD_+\mathbb{G}(u^{n+1/2}), D_-DD_+u^{n+1/2}\rangle}{\|D_-DD_+u^{n+1}\| + \norm{D_-DD_+u^n}}\\
   & \leq  \norm{D_-DD_+u^n} +  2\Delta t \norm{u^{n+1/2}}^2_{h^3}.
\end{align*}
For the lower-order derivatives, we follow the similar arguments as above, and using the conservation property, we end up with
\begin{align*}
   \norm{u^{n+1}}_{h^3} &\leq  \norm{u^n}_{h^3} +  2\Delta t \norm{u^{n+1/2}}^2_{h^3}.
\end{align*}
By applying the Lemma \ref{lemma:solv.lemma}, the following estimate holds
\begin{align*}
   \norm{u^{n+1}}_{h^3} &\leq  \norm{u^n}_{h^3} + \frac{\Delta t}{2}\big((K+1)\norm{u^n}_{h^3}\big)^2
\end{align*}
provided $\frac{\Delta t}{\Delta x} = {L}/(K\|u^n\|_{h^3})$.
Let us set $a_n= \norm{u^n}_{h^3}$. Then the above estimate can be represented as 
\begin{equation}\label{Auxi_eqn}
    a_{n+1} \leq a_n + \frac{\Delta t}{2}\left((K+1)a_n\right)^2.
\end{equation}
Let $y(t)$ solve the differential equation $$y'(t) = \frac{1}{2}(K+1)^2y(t)^2, \quad y(0) = \|u_0\|_{h^3}.$$ 
It is straightforward to observe that the solution $y$ is blowing up at time ${T}_\infty = (\frac{1}{2}(K+1)^2\|u_0\|_{h^3})^{-1}$. Moreover, $y(t)$ is strictly increasing whenever $t<T_\infty$. If we choose $t<T<T_\infty$, then $y(t)\leq y(T) =:M$ provided $\frac{\Delta t}{\Delta x} = {L}/(K\|u^n\|_{h^3})$ holds. 

Now we claim that $a_n\leq y(t_n)\leq M$ for $t_n\leq T$. This claim is true for $n=0$. Let us assume that it is true for $n=0,\dots,N$. Then
\begin{equation*}
    \begin{split}
        a_{N+1} \leq a_N +\frac{\Delta t}{2}((K+1)a_N)^2 \leq &  y(t_N) +\frac{\Delta t}{2}((K+1)y(t_N))^2\\
        \leq & y(t_N) +\int_{t_N}^{t_{N+1}}\frac{1}{2}((K+1)y(t_N))^2\,dt\\
        \leq &   y(t_N) +\int_{t_N}^{t_{N+1}}y'(s)\,ds = y(t_{N+1}).
    \end{split}
\end{equation*}
This proves that $a_n \leq y(T) = M$ for all $n$ such that $t_n\leq T$. Thus $\|u^n\|_{h^3}\leq M \leq C(\|u_0\|_{h^3})$. This completes the proof.
\end{proof}

We seek for a temporal bound of the approximations which is crucial for the convergence proof. From the scheme \eqref{CNFDscheme}, we can rewrite it by taking the $\ell^2$-norm on both sides:
\begin{equation*}
\norm{D_+^tu^n} \leq \norm{\mathbb{G}(u^{n+1/2})} + \norm{D_-DD_+u^{n+1/2}}.
\end{equation*}
Since we have the estimate $\norm{\mathbb{G}(u^{n+1/2})} \leq \norm{\Tilde{u}^{n+1/2}}_\infty\norm{Du^{n+1/2}}\leq \norm{u^{n+1/2}}_{h^3}\leq C$ and taking into account the Lemma \ref{lemma:Stablemma}, we deduce
\begin{equation}\label{Temp_Bound}
    \norm{D_+^tu^n} \leq C.
\end{equation}
 
Now we will state the main result of this section. We shall establish the convergence of the approximate solution to the classical solution for $t<T$. We follow the approach of Sjöberg \cite{sjoberg1970korteweg} to prove the following result.
\begin{theorem}\label{Conv_THM}
Suppose the initial data $u_0\in H^3(\R)$.
Let $\{u^n\}$ be a sequence of difference approximations obtained by the scheme \eqref{CNFDscheme}. Furthermore assume that $\Delta t = \mathcal{O}(\Delta x)$. Then there exists a finite time $T$ and a constant $C$, depending only on $\|u_0\|_{h^3}$ such that 
    \begin{align}
        \label{bb1}\|u_{\Delta x}(\cdot,t)\|_{L^2(\mathbb{R})} &\leq \|u_0\|_{L^2(\mathbb{R})},\\
       \label{bb2} \|\partial_xu_{\Delta x}(\cdot,t)\|_{L^2(\mathbb{R})} &\leq C,\\
         \label{bb3} \|\partial_tu_{\Delta x}(\cdot,t)\|_{L^2(\mathbb{R})} &\leq C,\\
          \label{bb4} \left\|\partial^3_xu_{\Delta x}(\cdot,t)\right\|_{L^2(\mathbb{R})} &\leq C.
    \end{align}
    Moreover, the sequence of approximate solutions $\{u_{\Delta x}\}_{{\Delta x}\geq 0}$ converges uniformly in $C(\mathbb{R}\times [0,T]) $ to the unique solution of the KdV equation \eqref{kdveqn} as $\Delta x\xrightarrow[]{} 0$.
\end{theorem}
To prove the Theorem \ref{Conv_THM}, we need to define a weak solution of \eqref{kdveqn}.
\begin{definition}\label{defn}
    Let $Q>0$ and $u_0\in L^2(\mathbb{R})$. Then $u\in L^\infty(0,T;L^2(\R))$ is said to be a weak solution of \eqref{kdveqn} if 
    \begin{equation}\label{weaks}
        \int_0^T \int_{\mathbb{R}}\left(\varphi_t u + \varphi_x\frac{u^2}{2} +\varphi_{xxx}u\right)\,dx\,dt + \int_{\mathbb{R}}\varphi(x,0)u_0(x)\,dx = 0,
    \end{equation}
    for all $\varphi\in C_c^\infty((-Q,Q)\times[0,T)) $.
\end{definition}
\begin{proof}[Proof of Theorem \ref{Conv_THM}]
We observe that the approximation $u_{\Delta x}$ defined by \eqref{interCN} is continuous everywhere and continuously differentiable in space. Hence for $x\in [x_{j},x_{j+1})$ and $t\in [t_n,t_{n+1}]$, we have
    \begin{align*}
        \partial_x u_{\Delta x}(x,t) =& D_+u_j^n + (x-x_j)D_+D_-u^n_j
        +\frac{1}{2}(x-x_j)^2D_+DD_-u^n_j \\
        &+ (t-t_n)D_+^t\left(D_+u_j^n + (x-x_j)D_+D_-u^n_j 
        +\frac{1}{2}(x-x_j)^2D_+DD_-u^n_j\right),\\
        \partial^2_x u_{\Delta x}(x,t) =&D_+D_-u^n_j
        +(x-x_j)D_+DD_-u^n_j+ (t-t_n)D_+^t\left(D_+D_-u^n_j 
        +(x-x_j)D_+DD_-u^n_j\right),\\
        \partial^3_x u_{\Delta x}(x,t) =& D_+DD_-u^n_j + (t-t_n)D_+^t\left(D_+DD_-u^n_j\right),\\
        \partial_t u_{\Delta x}(x,t) =& D_+^tu^n(x).
    \end{align*}
These identities imply that the estimates \eqref{bb1}-\eqref{bb4} hold for $t\leq T$.
Due to \eqref{Temp_Bound}, we have the boundedness of $\partial_t u_{\Delta x}$. It further implies that $u_{\Delta x}\in \text{Lip}([0,T];L^2(\mathbb{R}))$. Incorporating \eqref{bb1}-\eqref{bb3} and applying the Arzela-Ascoli theorem, we conclude that the set $\{u_{\Delta x}\}_{\Delta x>0}$ is sequentially compact in $C([0,T];L^2(\mathbb{R}))$. Then there exists a subsequence $\{u_{\Delta x}\}_{j\in\mathbb{N}}$ which converges uniformly in $C([0,T];L^2(\mathbb{R}))$ to some function $u$. 
We claim that the limit $u$ is a weak solution of the Cauchy problem \eqref{kdveqn} in the sense of Definition \ref{defn}.
Hereby we can apply the straightforward modification of the proof of Lax-Wendroff type result in Holden et al. \cite{holden1999operator} to conclude that $u$ is a weak solution.
The bounds \eqref{bb1}-\eqref{bb4} imply that $u$ is eventually a strong solution such that \eqref{kdveqn} holds as an $L^2$-identity. We conclude that the limit $u$ is the unique solution of the KdV equation \eqref{kdveqn} incorporating the initial data $u_0$. Hence the result follows.
\end{proof}
\subsection{Convergence analysis with $L^2$ initial data}
In this section, we shall establish the convergence of difference approximations generated by the same devised scheme \eqref{CNFDscheme} to a weak solution of the KdV equation \eqref{kdveqn} under the condition that the initial data lacks regularity, i.e. $u_0\in L^2(\mathbb{R})$. Given the inherent lack of smoothness in the initial data, previous conventional estimates cannot be readily applied. Nevertheless, we will adopt the \emph{Kato's theory of smoothing effects} \cite{kato1983cauchy} which asserts that even in the presence of non-smooth initial data, the solution exhibits localized smoothing due to its dispersive nature. This property will enable us to derive estimates which are pivotal for the convergence analysis. In particular, we note that such smoothing effects do not hold for hyperbolic equations. To elaborate further, 
Kato \cite{kato1983cauchy} established that the solution of \eqref{kdveqn} satisfies the following inequality (refer to \cite{kenig1991well, holden2015convergence})
\begin{equation}\label{smooth_effect}
    \left(\int_{-T}^{T} \int_{-R}^{R}|u_x|^2\,dx\,dt\right)^{1/2}\leq C(T,R), \qquad T,~R>0.
\end{equation}
We briefly explain our approach towards the convergence analysis. With the help of \eqref{smooth_effect}, we will show that $u_{\Delta x}\in W$, where the function space $W$ is given by 
\begin{equation*}
     W = \left\{w\in L^2(0,T;H^1(-Q,Q)) \, \big| \, w_t\in L^{4/3}(0,T;H^{-3}(-Q,Q))\right\}.
\end{equation*}
Afterwards, our approach will rely on the Aubin-Simon compactness lemma \cite{dutta2016note} which ensures compactness of the sequence of approximate solutions in the space $L^2(0,T;L^2(-Q,Q))$. 
It is worth noting that a similar methodology was employed in \cite{holden2015convergence} concerning the Euler implicit finite difference scheme. However, we wish to establish improved convergence estimates for our conservative scheme \eqref{CNFDscheme}.

We begin by defining a non-negative function \(p\) that is both smooth and compactly supported. Given \(R > 0\) be a fixed constant,
we define the function \(p(x)\) as follows:
\begin{align}\label{defn_p}
p(x) = 1 + \int_{-\infty}^x \omega(s)^2 \,ds,
\end{align}
where $\omega$ is a non-negative compactly supported smooth function with the following properties: \(\omega(x) = 1\) for \(|x| < R\), \(\omega(x) = 0\) for \(|x| \geq R+1\), and \(0 \leq \omega(x) \leq 1\).
The construction of $p(x)$ ensures that it remains bounded. More precisely, it satisfies
\begin{align*}
    &1 \leq p(x) \leq 1+(2+2R), \quad p_x(x) = 1~~\text{for } |x| < R, \\
    & p_x(x) = 0~~\text{for }|x| \geq R+1, \quad 0\leq p_x(x)\leq 1, \quad \sqrt{p_x}\in C_c^\infty(\R).
\end{align*}
Corresponding to the non-negative function $p$, we define a weighted inner product and associated norm
\begin{equation*}
    \langle u, v \rangle_p: = \langle u, pv\rangle =  \Delta x \sum_{j}p_ju_jv_j, \qquad \norm{u}_p^2 := \langle u, u \rangle_p,
\end{equation*}
where $p_j = p(x_j)$. Thanks to the properties of $p(x)$, we have $$\norm{u}^2\leq \norm{u}_p^{2} \leq (3+2R)\norm{u}^2. $$
We seek to estimate $\langle D_-DD_+ u , u \rangle_p$ which will be essential in our further analysis. 
\begin{lemma}
    Let the function $p(x)$ be defined by \eqref{defn_p}. Then there holds
    \begin{align}\label{u_xxxp}
        \langle D_-DD_+ u , u \rangle_p \geq -C_R \norm{u}^2_p,
    \end{align}
where the constant $C_R$ is given by
\begin{equation}\label{C_R}
    C_R = \max{\big\{\norm{p}_{L^\infty(\R)} , \norm{p_x}_{L^\infty(\R)}, \norm{p_{xx}}_{L^\infty(\R)}, \norm{p_{xxx}}_{L^\infty(\R)}\big\}}.
\end{equation}
\end{lemma}
\begin{proof}
Since all the difference operators commute with each other, we have
\begin{align*}
    \langle D_-DD_+ u , u \rangle_p =& \langle D_-DD_+ u , up \rangle \\=& \frac{1}{2}\langle D u , D_+D_-(u p)\rangle + \frac{1}{2}\langle D u , D_-D_+(u p)\rangle\\
    =&  \frac{1}{2}\langle D u , pD_+D_-u + 2DuD_+p + S^-uD_+D_-p\rangle \\& +\frac{1}{2}\langle D u , pD_+D_-u + 2DuD_-p + S^+uD_+D_-p\rangle \\
    =& \langle D u , pD_+D_-u + 2DuDp + \bar uD_+D_-p\rangle \\
     =& \langle D uD_+D_-u , p \rangle+ 2\langle (D u)^2,Dp\rangle + \langle \bar uD u,D_+D_-p\rangle\\
     =&  \frac{1}{2}\langle D_+(D_-u)^2 , p \rangle+ 2\big\langle (D u)^2,Dp\big\rangle + \frac{1}{2}\langle D u^2,D_+D_-p\rangle\\
     =& 2\big\langle (D u)^2,Dp\big\rangle -\frac{1}{2}\big\langle (D_-u)^2 , D_-p \big\rangle - \frac{1}{2}\langle  u^2,D_+DD_-p\rangle.
\end{align*}
Using the properties of $p$ yields
\begin{align}\label{D3esti}
    2\big\langle (D u)^2,Dp\big\rangle -\frac{1}{2}\big\langle (D_-u)^2 , D_-p \big\rangle &\geq 2\Delta x \sum_{|j\Delta x|\leq R-1} (Du_j)^2 - \frac{1}{2}\Delta x \sum_{|j\Delta x|\leq R-1} (D_-u_j)^2\\&\geq \frac{3}{2}\Delta x \sum_{|j\Delta x|\leq R-1} (D_-u_j)^2= \frac{3}{2}\Delta x \sum_{|j\Delta x|\leq R-1} (D_+u_j)^2  \geq 0. \nonumber
\end{align}
Afterwards, employing \eqref{C_R} in the above estimates we have
\begin{align*}
    \langle D_-DD_+ u , u \rangle_p \geq -C_R \norm{u}^2 \geq -C_R \norm{u}^2_p.
\end{align*}
Hence the result follows.
\end{proof}
We state and prove the following lemma to ensure the solvability of the scheme \eqref{CNFDscheme} whenever initial data $u_0\in L^2(\R)$.
\begin{lemma}\label{lemma:solv.lemmaL2}
Consider the iterative scheme given by \eqref{iterr}. Assume that the following CFL condition satisfies
    \begin{equation}\label{CNCFLL2}
        \lambda \leq \frac{7L}{8K\norm{u^n}_p},
    \end{equation}
where $\lambda = \frac{\Delta t}{\Delta x^{3/2}}$ and $K =\frac{5-L}{1-L}>5$ be a constant with $0<L<1$. Then there exists a function $u^{n+1}$ that solves equation \eqref{CNFDscheme} and $\lim_{\ell\xrightarrow{}\infty}w^{\ell} = u^{n+1}$. Moreover, the following estimate holds:
    \begin{equation}\label{stabn2}
        \norm{u^{n+1}}_p \leq K \norm{u^n}_p.
    \end{equation}
\end{lemma}
\begin{proof} 
Let us set $\Delta w^\ell := w^{\ell+1} - w^\ell $. Then \eqref{CNFDscheme} can be represented as
\begin{equation}\label{iterrwl2}
    \left(1+\frac{\Delta t}{2} D_-DD_+\right)\Delta w^\ell = -\Delta t \left[\mathbb{G}\left(\frac{u^n+w^\ell}{2}\right) - \mathbb{G}\left(\frac{u^n+w^{\ell-1}}{2}\right)\right]=:-\Delta t \Delta \mathbb{G}.
\end{equation}
Taking inner product of \eqref{iterrwl2} with $p\Delta w^\ell$, we obtain
\begin{equation*}
   \begin{split}
    \norm{\Delta w^\ell}_p^2 +\frac{\Delta t}{2}\big\langle D_-DD_+ \Delta w^\ell, \Delta w^\ell \big\rangle_p  &= -\Delta t \big\langle \Delta \mathbb{G}, \Delta w^\ell\big\rangle_p
     \leq  \Delta t  \norm{\Delta \mathbb{G}}_p \norm{\Delta w^\ell}_p.
     \end{split}
\end{equation*}
Using \eqref{u_xxxp}, we deduce
\begin{equation}\label{temp:L2_1}
    \left(1-C_R\frac{\Delta t}{2}\right)\norm{\Delta w^\ell}_p^2 \leq  \Delta t  \norm{\Delta \mathbb{G}}_p \norm{\Delta w^\ell}_p.
\end{equation}
As earlier, we observe that
\begin{equation*}
    \Delta \mathbb{G} =  -\frac{1}{4}\left[ \widetilde{\Delta w^{\ell-1}} D(u^n+w^{\ell-1}) + \widetilde{(u^n+w^\ell)}D{\Delta w^{\ell-1}}\right].
\end{equation*}
The terms involved in $\|\Delta \mathbb{G}\|_p$ can be estimated by applying the Lemma A.1 in  \cite{holden2015convergence} along with identity \eqref{D3z_1z_2} 
\begin{align*}
    \norm{\widetilde{\Delta w^{\ell-1}} D(u^n+w^{\ell-1})}_p &\leq \norm{\sqrt{p}D(u^n+w^{\ell-1})}_{\infty}\norm{\Delta w^{\ell-1}}_p\\
    &\leq \frac{1}{\Delta x^{3/2}}  \big( \norm{u^n}_p + \norm{w^{\ell-1}}_p \big)\norm{\Delta w^{\ell-1}}_p\\
    &\leq \frac{2}{\Delta x^{3/2}} \max \big\{ \norm{u^n}_p , \norm{w^{\ell-1}}_p \big\}\norm{\Delta w^{\ell-1}}_p,
\end{align*}
and similarly
\begin{equation*}
    \norm{\widetilde{(u^n+w^\ell)}D{\Delta w^{\ell-1}}}_p \leq \frac{2}{\Delta x^{3/2}} \max \big\{ \norm{u^n}_p , \norm{w^{\ell}}_p \big\}\norm{\Delta w^{\ell-1}}_p.
\end{equation*}
Combining the above estimates and choosing $\Delta t $ sufficiently small such that $C_R\Delta t \leq \frac{1}{4}$, \eqref{temp:L2_1} reduces to
\begin{equation}\label{hpboundd2}
     \norm{\Delta w^\ell}\leq \frac{8}{7}\lambda \max \big\{ \norm{u^n}_p , \norm{w^{\ell}}_p ,\norm{w^{\ell-1}}_p \big\}\norm{\Delta w^{\ell-1}}_p.
\end{equation}
Afterwards, we use the induction argument to prove that the sequence $\{w^\ell\}$ is Cauchy. We know that $w^1$ satisfies 
\begin{equation}\label{w^12}
     w^{1} = u^n - \Delta t \mathbb{G}(u^n) - \Delta t D_-DD_+\Big(\frac{u^n+w^1}{2} \Big).
\end{equation}
By taking the inner product with $p(u^n+w^1)$ we get
\begin{equation*}
        \|w^1\|_p^2 + \frac{\Delta t}{2}\big\langle D_-DD_+(u^n+w^1) , p(u^n+w^1)\big\rangle = \|u^n\|_p^2 -\Delta t \big\langle\mathbb{G}(u^n) , p(u^n+w^1)\big\rangle
\end{equation*}
which further becomes
\begin{equation*}
        \|w^1\|_p^2 -C_R \frac{\Delta t}{2}\norm{u^n+w^1}_p^2 \leq \|u^n\|_p^2 +\Delta t^2 \norm{\mathbb{G}(u^n)}^2_p +\frac{1}{4} \norm{u^n+w^1}_p^2.
\end{equation*}
Taking into account the estimate \(\norm{\mathbb{G}(u^n)}^2_p = \norm{\Tilde{u}^nDu^n}^2_p \leq \norm{\sqrt{p}u^n}_\infty^2\norm{Du^n}_p^2\leq \frac{1}{(\Delta x^{3/2})^2}\norm{u^n}^4_p\), we have
\begin{align*}
    \left(\frac{1}{2} -C_R\Delta t \right)\|w^1\|_p^2 &\leq \left(\frac{3}{2} +C_R\Delta t \right)\|u^n\|_p^2 +\lambda^2 \norm{u^n}^4_p
\end{align*}
which turns into
\begin{align*}
    \norm{w^1}_p^2 &\leq 4\left(\frac{7}{4} +\lambda^2 \|u^n\|_p^2\right)\|u^n\|_p^2
\end{align*}
provided $\Delta t $ is sufficiently small such that $C_R\Delta t \leq \frac{1}{4}$.
Choice of $L$ and $K$ implies
\begin{equation*}
     2\left(\frac{7}{4} +\lambda^2 \|u^n\|_p^2\right)^{1/2}\leq 3.
\end{equation*}
Finally, we end up with
\begin{equation*}
    \|w^1\|_p \leq 3 \|u^n\|_p \leq K \|u^n\|_p. 
\end{equation*}
For the induction argument, we assume
\begin{align}
       \label{fis1} \|w^\ell\|_p &\leq K \|u^n\|_p \qquad \text{for } \quad \ell =1,...,m,\\
       \label{secd2} \|\Delta w^\ell\|_p &\leq L \|\Delta w^{\ell-1}\|_p \qquad \text{for } \quad \ell  =2,...,m.
\end{align}
We have estimated \eqref{fis1} for $m=1$. We show the estimate \eqref{secd2} for $m=2$. Due to \eqref{hpboundd2} we have the following estimate
\begin{equation*}
    \|\Delta w^{2}\|_p \leq \frac{8}{7}\lambda\max \big\{\|u^n\|_p , \|w^1\|_p \big\} \|\Delta w^1\|_p \leq \frac{8}{7}\lambda K\|u^n\|_p\|\Delta w^1\|_p \leq L \|\Delta w^1\|_p.
\end{equation*}
After that we show \eqref{fis1} for $m>1$,
\begin{align*}
    \|w^{m+1}\|_{p} &\leq \sum_{\ell=0}^{m}\|\Delta w^\ell\|_{p} +\|u^n\|_{p} \leq \|w^1-u^n\|_{p}\sum_{\ell=0}^{m} L^\ell + \|u^n\|_{p}\\
    &\leq \left(\|w^1\|_{p} +\|u^n\|_{p}\right) \frac{1}{1-L} + \|u^n\|_{p} \leq \frac{5-L}{1-L} \|u^n\|_{p} = K\|u^n\|_{p}.
\end{align*}
Moreover, we also get 
\begin{equation*}
    \|\Delta w^{m+1}\|_{p} \leq \frac{8}{7}\lambda K \|u^n\|_{p} \|\Delta w^m\|_{p}\leq L\|\Delta w^m\|_{p}
\end{equation*}
provided \eqref{CNCFL} holds.
Summing up all the above estimates, we obtain the desired estimate \eqref{stabn2} and by \eqref{secd2}, the sequence $\{w^\ell\}$ is Cauchy, hence converges to $u^{n+1}$. This completes the proof.
\end{proof}
\begin{remark}
It is observed that the CFL condition \eqref{CNCFLL2}, imposed in the Lemma \ref{lemma:solv.lemmaL2}, depends on the approximate solution at time step $n$. However, we require the CFL condition to depend on the initial data $u_0$. Hence we need to derive a priori bound for the approximate solution $u^n$. 
\end{remark}
\begin{lemma}\label{lemma:StablemmaL2}
    Let $u_0\in L^2(\R)$. Assume that $\Delta t$ satisfies 
    \begin{equation}\label{CFLCNL22}
        \lambda =\frac{\Delta t}{\Delta x^{3/2}}\leq \frac{7L}{8KM}
    \end{equation}
for some $M$ which only depends on $\|u_0\|_{L^2(\R)}$. Then there exist a finite time $T>0$ and a constant $C$ both depending on $\|u_0\|_{L^2(\R)}$ such that 
    \begin{equation}\label{stabcnL2}
        \|u^n\| \leq C\big(\norm{u^0} ,R\big) \qquad \text{for }~ t_n\leq T.
    \end{equation}
Moreover, there holds $H^1_{\text{loc}}$ bound of the approximate solution
    \begin{equation}\label{H^1B_}
        \Delta t \Delta x \sum_{n=0}^{N-1} \sum_{|j\Delta x|\leq R-1}\big(D_+u_j^n\big)^2 \leq C.
    \end{equation}
\end{lemma}
\begin{proof}
Taking the inner product with $pu^{n+1/2}$ of \eqref{CNFDscheme} and further using the estimates \eqref{D3esti} and \eqref{u_xxxp}, we get
    \begin{align}
       \frac{1}{2}\norm{u^{n+1}}_p^2 +\frac{3}{2}\Delta t\Delta x \sum_{|j\Delta x|\leq R-1} \big(D_+u^{n+1/2}_j\big)^2  &- C_R\frac{\Delta t}{2} \norm{u^{n+1/2}}^2_p\nonumber \\& \label{H^1bo}\leq \frac{1}{2}\norm{u^n}_p^2 - \Delta t \langle \mathbb{G}(u^{n+1/2}), u^{n+1/2}\rangle_p.
       \end{align}
In order to estimate $ \langle \mathbb{G}(u^{n+1/2}), u^{n+1/2}\rangle_p$, we drop the superscript $n+1/2$ for the moment
\begin{align*}
     \langle \mathbb{G}(u), u\rangle_p &=  \langle \Tilde{u}Du, up\rangle = \frac{1}{3} \big\langle (S^+u+u+S^-u)Du, up\big\rangle
     = \frac{1}{3} \langle uDu, up\rangle + \frac{1}{3} \langle Du^2, up\rangle\\
     &= -\frac{1}{3} \langle u, D(u^2p)\rangle + \frac{1}{3} \langle Du^2, up\rangle
     = -\frac{1}{3} \langle \overline{u^2}u, Dp\rangle -\frac{1}{3} \langle u, \overline pDu^2\rangle+ \frac{1}{3} \langle Du^2, up\rangle \\
     &= -\frac{1}{3} \langle \overline{u^2}u, Dp\rangle- \frac{1}{3} \langle uDu^2, \overline p-p\rangle
     = -\frac{1}{3} \langle \overline{u^2}u, Dp\rangle- \frac{1}{3}\frac{\Delta x^2}{2} \langle uDu^2, D_-D_+ p\rangle\\
     &\leq \frac{\Delta x}{3}\norm{u}_\infty^2\norm{u}\norm{D_-D_+p}+  \frac{1}{3}\frac{\Delta x^2}{2} \norm{Du^2}_\infty\norm{u}\norm{D_-D_+p}\\
     &\leq \frac{1}{3}C_R \norm{u}_p^3+\frac{1}{3}\frac{\Delta x}{2} \norm{u}^2_\infty\norm{u}\norm{D_-D_+p}\\
     &\leq \frac{1}{2}C_R \norm{u}_p^3.
\end{align*}
Dropping the positive second term from the left hand side in \eqref{H^1bo} and substituting the above estimate, we have
\begin{align*}
    \norm{u^{n+1}}_p - \norm{u^n}_p &\leq 2C_R\frac{\Delta t}{2}  \frac{\norm{u^{n+1/2}}^2_p + \norm{u^{n+1/2}}_p^3 }{\norm{u^{n+1}}_p + \norm{u^n}_p }\\
    &\leq C_R \frac{\Delta t}{2}  \big(\norm{u^{n+1/2}}_p + \norm{u^{n+1/2}}_p^2\big)\\
    &\leq C_R\frac{\Delta t}{2}  K\big(\norm{u^{n}}_p + K\norm{u^{n}}_p^2\big),
\end{align*}
where we have used \eqref{stabn2}. Set $a_n = \norm{u^{n}}_p$. Then the above estimate can be represented by
       \begin{equation}\label{diffeqn_an}
           a_{n+1} \leq a_n + C_R\frac{\Delta t}{2}  K\big(a_n + Ka_n^2\big).
       \end{equation}
Let $y(t)$ satisfy the differential equation
\begin{equation*}
    y'(t) = C_R\frac{K}{2}\big(y(t)+Ky(t)^2\big),\quad y(0) = \norm{u_0}.
\end{equation*}
It is straightforward to observe that the solution of this equation will blow up at finite time, say  $\widehat{T}$. Also for $t<T:=\frac{\widehat{T}}{2}$, $y(t)$ is strictly increasing and convex. 

Next, we claim that $a_n\leq y(t_n)$ for all $t_n\leq T$ under the assumption that \eqref{CFLCNL22} holds.
We proceed by mathematical induction, clearly $a_0 \leq y(0)$ holds for $n=0$. Let us assume that the claim is true for $n=0,1,2,\dots,m$. Since $0<a_m \leq M:=y(T)$, then \eqref{CFLCNL22} implies that $\lambda$ also satisfies the CFL condition \eqref{CNCFLL2} in Lemma \ref{lemma:solv.lemmaL2}. Applying the Lemma \ref{lemma:solv.lemmaL2}, we have $a_{m+1/2}\leq K a_m$. Then \eqref{diffeqn_an} yields
\begin{align*}
    a_{m+1} &\leq y(t_m) +C_R\frac{\Delta t}{2}K\big(y(t_m) + K y(t_m)^2\big)\\
            &\leq y(t_m) +\int_{t_m}^{t_{m+1}}C_R\frac{K}{2}\big(y(t_m) + K y(t_m)^2\big)\,dt\\
            &\leq y(t_m) +\int_{t_m}^{t_{m+1}}C_R\frac{K}{2}\big(y(t) + K y(t)^2\big)\,dt\\
            &\leq y(t_m) +\int_{t_m}^{t_{m+1}}y'(t)\,dt \leq y(t_{m+1}).
\end{align*}
Hence the claim is established. Since $1\leq p \leq (3+2R)$, we have the required estimate
\begin{equation*}
    \norm{u^n} \leq M \leq C\big(\norm{u^0} ,R \big).
\end{equation*}
Now dropping the first term from the left-hand side in \eqref{H^1bo} and summing it over $n$, we have
\begin{equation*}
   \Delta t\Delta x \sum_{n=0}^{N-1}\sum_{|j\Delta x|\leq R-1} \big(D_+u^{n+1/2}_j\big)^2 \leq C\big(\norm{u^0} ,R \big).
\end{equation*}
Hence $H^1_{\text{loc}}$-estimate is obtained.
\end{proof}
 Before proceeding with the convergence proof, we estimate the temporal derivative of the approximate solution $u_{\Delta x}$.
 \begin{lemma}\label{lemma:PartialT}
Let $\{u^n\}$ be a sequence of approximate solutions generated by the numerical scheme \eqref{CNFDscheme}, $\{u_{\Delta x}\}$ is obtained by interpolation \eqref{interCN}. Furthermore, assume that the hypothesis of Lemma \ref{lemma:StablemmaL2} holds. Then the following estimate holds
     \begin{equation}\label{Time_boundL^2}
         \norm{\partial_t u_{\Delta x}}_{L^{4/3}(0,T;{H^{-3}(-Q,Q)})} \leq C(\norm{u^0},R),
     \end{equation}
     where $Q = R-1$.
 \end{lemma}
 \begin{proof}
From the scheme \eqref{CNFDscheme} we have 
     \begin{equation}\label{D_tn}
    D^t_+u_j^{n} =  \Tilde{u}_j^{n+1/2} Du_j^{n+1/2} - D_-DD_+u^{n+1/2}_j,\qquad n\in\mathbb{N}_0,\hspace{0.1cm}j\in\mathbb{Z}.
\end{equation}
Since $\partial_t u_{\Delta x}(x,t) =D^t_+u^{n}(x) $ for $x\in[x_j,x_{j+1})$ and $t\in[t_n,t_{n+1}]$,
we estimate the terms in the right-hand side of \eqref{D_tn} by repetitive use of the H{\"o}lder's inequality and simple use of truncation analysis. Let $\phi\in H^3_0(-Q,Q)$ be any test function 
\begin{align*}
\bigg| \int_{-Q}^{Q} & (D_-DD_+u^{n+1/2})\phi(x) \,dx \bigg|\\
 \leq & \sum_{|j\Delta x|\leq Q} |Du_j^{n+1/2}|\int_{x_j}^{x_{j+1}}|\phi''(x)|\,dx 
+ \sum_{|j\Delta x|\leq Q} |Du_j^{n+1/2}|\int_{x_j}^{x_{j+1}}|D_-D_+\phi(x) - \phi''(x)|\,dx\\
\leq &  \sum_{|j\Delta x|\leq Q} |Du_j^{n+1/2}|\sqrt{\Delta x} \norm{\phi''}_{L^2(x_j,x_{j+1})} \\
&\qquad + \frac{1}{\Delta x^2}\sum_{|j\Delta x|\leq Q} \big|Du_j^{n+1/2} \big|
\left(\int_{x_j}^{x_{j+1}}\int_{x}^{x+\Delta x}\int_{z-\Delta x}^{z}\int_{x}^{\tau} |\phi'''(\theta)|\,d\theta\,d\tau\,dz\,dx\right)\\
\leq & \left(\sum_{|j\Delta x|\leq Q} \Delta x \big|Du_j^{n+1/2} \big|^2\right)^{1/2} \left(\sum_{|j\Delta x|\leq Q}\norm{\phi''}^2_{L^2(x_j,x_{j+1})}\right)^{1/2} \\
&\qquad + \Delta x^{3/2}\sum_{|j\Delta x|\leq Q} \big|Du_j^{n+1/2} \big|\norm{\phi'''(\theta)}_{L^2(x_j,x_{j+1})}\\
\leq& \norm{D_+u^{n+1/2}}_{L^2(-Q,Q)}\norm{\phi''}_{L^2(-Q,Q)} + \Delta x\norm{D_+u^{n+1/2}}_{L^2(-Q,Q)}\norm{\phi'''}_{L^2(-Q,Q)}\\
\leq &C\norm{D_+u^{n+1/2}}_{L^2(-Q,Q)}.
\end{align*}
Therefore, we derive the following estimate
\begin{align*}
    \norm{D_-DD_+u^{n+1/2}}_{H^{-3}(-Q,Q) }\leq C\norm{D_+u^{n+1/2}}_{L^2(-Q,Q)}.
\end{align*}
By using the Lemma \ref{lemma:StablemmaL2} and H{\"o}lder's inequality, it provides
\begin{align*}
    \Delta t\sum_{n=0}^{N}\norm{D_-DD_+u^{n+1/2}}_{H^{-3}(-Q,Q) }^{4/3}&\leq CT^{1/3}\left(\Delta t\Delta x\sum_{n=0}^{N}\sum_{|j\Delta x|\leq Q}\Big|D_+u_j^{n+1/2} \Big|^2\right)^{2/3}\\
    &\leq C\big(\norm{u^0}_{L^2(\R)},R).
\end{align*}
Let us define a $C_c^\infty$ cut-off function $\xi$ such that $0\leq\xi\leq 1$ and
\begin{align*}
   \xi(x) = 
    \begin{cases}
        1, \qquad |x|\leq Q,\\
        0,\qquad |x| \geq Q+1.\\
    \end{cases}
\end{align*}
Set $\xi_{j} = \xi(x_j)$ and consider $\Tilde{u}_j^{n+1/2}Du^{n+1/2}_j$ as a piecewise constant function on $(t_n,t_{n+1})\times (x_j,x_{j+1})$. Using the H{\"o}lder's inequality, we deduce
\begin{align*}
    \Delta t\sum_{n=0}^{N-1}\biggr(\Delta x & \sum_{|j\Delta x|\leq Q}\big|\xi_j\Tilde{u}_j^{n+1/2}Du^{n+1/2}_j\big|^2\biggr)^{2/3} \\ 
    & \leq
     \Delta t\sum_{n=0}^{N-1}\norm{\xi\Tilde{u}^{n+1/2}}_{\infty}^{4/3}\left(\Delta x\sum_{|j\Delta x|\leq Q}\big|Du_j^{n+1/2}\big|^2\right)^{2/3} \\
     \leq & \left(\Delta t \sum_{n=0}^{N-1} \norm{\xi\Tilde{u}^{n+1/2}}^4_{\infty}\right)^{1/3}
     \left(\Delta t \Delta x\sum_{n=0}^{N-1} \bigg(\sum_{|j\Delta x|\leq Q}\big|Du_j^{n+1/2}\big|^2\bigg)\right)^{2/3}\\
     \leq &  C\big(\norm{u^0},R)\left(\Delta t \sum_{n=0}^{N-1} \norm{\xi\Tilde{u}^{n+1/2}}_{\infty}^4\right)^{1/3}.
\end{align*}
By the discrete Sobolev inequality $\norm{u}_{\infty}\leq 2\norm{u}^{1/2}\norm{D_+u}^{1/2}$, and using the properties of $\xi$ and the Lemma \ref{lemma:StablemmaL2}, we derive
\begin{align*}
    \Delta t \sum_{n=0}^{N-1} \norm{\xi u^{n+1/2}}_{\infty}^4 &\leq  2\Delta t \sum_{n=0}^{N-1} \norm{\xi u^{n+1/2}}^2\norm{D_+(\xi u^{n+1/2})}^2\\
    &\leq 2\Delta t \sum_{n=0}^{N-1} \norm{\xi u^{n+1/2}}^2\left( \norm{u^{n+1/2}D_+\xi}^2 + \norm{S^+\xi D_+u^{n+1/2}}^2\right)\\
    &\leq 2\Delta t \sum_{n=0}^{N-1} \norm{ u^{n+1/2}}^4 +  2\Delta t \sum_{n=0}^{N-1}\norm{D_+u^{n+1/2}}^2\\
    & \leq C\big(\norm{u^0},R).
\end{align*}
It implies that $$\Tilde{u}^{n}Du^{n} \in L^{4/3}(0,T;L^2(-Q,Q)) \subset L^{4/3}(0,T;{H^{-3}(-Q,Q)}),$$  where $\Tilde{u}_j^{n}Du^{n}_j$ is a piecewise constant function in $[x_j,x_{j+1})\times[t_n,t_{n+1})$.
As a consequence, from \eqref{D_tn}, we conclude that $D^t_{+} u^n_j \in L^{4/3}(0,T;H^{-3}(-Q,Q))$. Hence the result follows.
\end{proof}
\begin{theorem} (Convergence to a weak solution)\\
Let $\{u_j^n\}$ be a sequence of difference approximations generated by \eqref{CNFDscheme} and $\{u_{\Delta x\}}$ be defined by \eqref{interCN}. Assume that all the hypothesis of Lemma \ref{lemma:StablemmaL2} holds and $\norm{u_0}_{L^2(\R)}$ is finite. Then there exists a constant $C$ such that 
    \begin{align}
        \label{L1}\norm{u_{\Delta x}}_{L^\infty(0,T;L^2(-Q,Q))} &\leq C,\\
        \label{L2}\norm{u_{\Delta x}}_{L^{2}(0,T;H^1(-Q,Q))} &\leq C,\\
        \label{L3}\norm{\partial_t u_{\Delta x}}_{L^{4/3}(0,T;H^{-3}(-Q,Q))} &\leq C,
    \end{align}
where $Q=R-1$. Furthermore, there exists a sequence $\{u_{\Delta x_j}\}$ converges to a weak solution $u\in L^2(0,T;L^2(-Q,Q))$ of \eqref{kdveqn}, i.e.  as $\Delta x_j \xrightarrow[]{j\xrightarrow{}\infty} 0$,
    \begin{equation}\label{Conv_L2}
        u_{\Delta x_j} \xrightarrow{} u \text{ strongly in } L^2(0,T;L^2(-Q,Q)).
    \end{equation}
\end{theorem}
\begin{proof}
   By \eqref{interCN}, we can write 
   \begin{equation*}
       u_{\Delta x}(x,t) = (1-\alpha_n(t))u^n(x) + \alpha_n(t)u^{n+1}(x), \qquad t\in [t_n,t_{n+1}), \quad n\in \N,
   \end{equation*}
where $\alpha_n(t) = \frac{t-t_n}{\Delta t} \in [0,1)$. Thus, by the Lemma \ref{lemma:StablemmaL2}, we have
   \begin{equation*}
       \norm{u_{\Delta x}(\cdot,t)}_{L^2(\R)} \leq \norm{u^n}_{L^2(\R)} + \norm{u^{n+1}}_{L^2(\R)}\leq C, \qquad \text{for all } t\in [t_n,t_{n+1}), \quad n\in \N.
   \end{equation*}
Hence \eqref{L1} is obtained.

To prove \eqref{L2}, we observe that
   \begin{equation*}
       \partial_x u_{\Delta x}(x,t) = (1-\alpha_n(t))\partial_x u^n(x) + \alpha_n(t)\partial_x u^{n+1}(x), \qquad t\in [t_n,t_{n+1}), \quad n\in \N.
   \end{equation*}
From \eqref{interCN} we obtain
   \begin{equation*}
       \partial_x u^n  = D_+u^n_j \qquad x\in[x_j,x_{j+1}),\quad j\in\Z.
   \end{equation*}
As a consequence, we have the following estimate
   \begin{align*}
       \norm{\partial_x u_{\Delta x}}_{L^{2}(0,T;L^2(-Q,Q))} =& \int_0^T \norm{\partial_x u_{\Delta x}(\cdot,t)}_{L^2(-Q,Q)}^2\,dt\\
        =& \sum_{n} \int_{t_n}^{t_{n+1}} \norm{(1-\alpha_n(t))\partial_x u^n + \alpha_n(t)\partial_x u^{n+1}}_{L^2(-Q,Q)}^2\,dt\\
       \leq& 2\sum_{n} \int_{t_n}^{t_{n+1}}(1-\alpha_n(t))^2 \norm{\partial_x u^n}_{L^2(-Q,Q)}^2 + \alpha_n(t)^2\norm{\partial_x u^{n+1}}_{L^2(-Q,Q)}^2\,dt\\
        \leq& 2\sum_{n} \frac{1}{\Delta t^2}\int_{t_n}^{t_{n+1}}(t_{n+1}-t)^2 \norm{\partial_x u^n}_{L^2(-Q,Q)}^2\,dt \\
         &+ 2\sum_{n}\frac{1}{\Delta t^2}\int_{t_n}^{t_{n+1}}(t-t_n)^2\norm{\partial_x u^{n+1}}_{L^2(-Q,Q)}^2\,dt\\
         \leq& 2\Delta t\sum_{n} \norm{\partial_x u^n}_{L^2(-Q,Q)}^2 
         +\norm{\partial_x u^{n+1}}_{L^2(-Q,Q)}^2\\
         \leq& 2\Delta t\sum_{n}\Delta x\sum_{|j\Delta x|\leq Q}( D_+u_j^n)^2
         +(D_+u_j^{n+1})^2\\
         \leq & C\big(\norm{u_0}_{L^2(\R)},R).
   \end{align*}
Hence \eqref{L2} holds. The estimate \eqref{L3} follows from the Lemma \ref{lemma:PartialT}. Since the space $H^1(-Q,Q)$ is compactly embedded in $L^2(-Q,Q)$ and $L^2(-Q,Q)$ is continuously embedded in $H^{-3}(-Q,Q)$, and the estimates \eqref{L1}-\eqref{L3} holds, we are in position to apply the Aubin-Lions-Simon compactness Lemma (one can refer to \cite{linares2014introduction}, \cite[Lemma 4.4]{holden2015convergence}). Hence we can extract a subsequence of $\{u_{\Delta x}\}$, still denoted by $\{u_{\Delta x}\}$, which converges strongly to a function $u$ in $L^{2}(0,T;L^2(-Q,Q))$.
The strong convergence allows us to pass the limit in the non-linearity. Thus we can apply the Lax-Wendroff like result of \cite{holden1999operator} to conclude that the limit $u$ is a weak solution of \eqref{kdveqn}. Hence we establish \eqref{Conv_L2}. This completes the proof.
\end{proof}
\section{Convergence rate of the scheme}\label{sec4}
In this section, we derive the error estimates in both time and space for the classical solution of KdV equation \eqref{kdveqn} assuming that $u_0$ is sufficiently smooth.
We start with the consistency of nonlinear term $uu_x$,  which is approximated by $\mathbb{G}(u) = \Tilde{u}Du$. It is observed that
\begin{align}
\mathbb{G}(u) = \Tilde{u}Du = \frac{1}{3}(S^+u + u +S^-u) Du = \frac{1}{3}uDu + \frac{1}{3}D(u^2).
\end{align}
A simple use of truncation error analysis and smoothness of $u$ implies that as $\Delta x \xrightarrow{} 0$,
\[\mathbb G(u) - uu_x =\mathcal{O}(\Delta x^2).\]
Now we derive the error estimate in the following theorem:
\begin{theorem}
Let $u$ be a classical solution of the KdV equation \eqref{kdveqn} and $u^n$ be the approximate solution generated by the devised scheme \eqref{CNFDscheme}. Then for $t_n\leq T$, there holds 
    \begin{equation}\label{err_est}
    \norm{u_n - u(t_n)} \leq C(u,T)\big(\Delta x^2+\Delta t^2\big).
    \end{equation}
\end{theorem}
\begin{proof}
Let us define $e_n: = u^n - u(t_n)$. Since $u$ is a classical solution of \eqref{kdveqn}, then from \eqref{CNFDscheme} we deduce
    \begin{equation}\label{Erroreqn}
        \frac{e_{n+1}-e_n}{\Delta t} + D_-DD_+e_{n+1/2}: = -\mathbb G(u^{n+1/2}) + \mathbb G(u(t_n+\Delta t/2)) - \mathcal G^n,
    \end{equation}
    where $\mathcal G^n$ is given by
    \begin{align*}
        \mathcal G^n:= \frac{u(t_{n+1})-u(t_n)}{\Delta t} + \mathbb G(u(t_n+\Delta t/2)) & + D_-DD_+u(t_n+\Delta t/2)\\
        & - \big(u_t + uu_x + u_{xxx}\big)(t_n+\Delta t/2).
    \end{align*}
Again by performing the truncation error analysis, we have 
    \begin{align}\label{temp:error_esti}
        &\norm{\frac{u(t_{n+1})-u(t_n)}{\Delta t}-u_t(t_n+\Delta t/2)} + \norm{\mathbb G(u(t_n+\Delta t/2)) - uu_x(t_n+\Delta t/2)}\leq  C(u)(\Delta x^2 + \Delta t^2),\nonumber\\
        &\norm{D_-DD_+u(t_n+\Delta t/2) -u_{xxx}(t_n+\Delta t/2) } \leq C(u)\Delta x^2.
    \end{align}
By taking the inner product with $e_{n+1/2}$, the equation \eqref{Erroreqn} becomes
    \begin{align}\label{temp:error_esti_1}
        \left\langle\frac{e_{n+1}-e_n}{\Delta t},e_{n+1/2}\right\rangle & + \langle D_-DD_+e_{n+1/2}, e_{n+1/2}\rangle  \\
        & = \langle\mathbb G(u(t_n+\Delta t/2))  -\mathbb G(u^{n+1/2}),e_{n+1/2}\rangle 
        - \langle\mathcal G^n , e_{n+1/2}\rangle
        := T_1 + T_2. \nonumber
    \end{align}
We estimate the terms $T_i,~i=1,2$. For convenience, dropping the subscript from $e_{n+1/2}$ and superscript from $u^{n+1/2}$ and denoting $u(t_n+\Delta t/2)=:u_{ex}$, we have
    \begin{align*}
     T_1  = & - \left\langle \mathbb{G}(u)-\mathbb{G}(u_{ex}),e\right\rangle = -\frac{1}{3}\left\langle D\big(u^2-u_{ex}^2\big),e\right\rangle - \frac{1}{3}\left\langle  uDu - u_{ex}Du_{ex},e\right\rangle\\
        =& -\frac{2}{3}\left\langle D\big(u_{ex}e\big),e\right\rangle -\frac{1}{3}\left\langle De^2,e\right\rangle - \frac{1}{3}\left\langle eDu_{ex},e\right\rangle - \frac{1}{3}\left\langle u_{ex}De,e\right\rangle - \frac{1}{3}\left\langle eDe,e\right\rangle\\
        = & \frac{1}{3}\left\langle u_{ex}De,e\right\rangle -\frac{1}{3}\left\langle eDu_{ex},e\right\rangle - \left\langle \mathbb G(e),e\right\rangle \\
        =&\frac{1}{3}\left\langle u_{ex},eDe\right\rangle -\frac{1}{3}\left\langle Du_{ex},e^2\right\rangle
        \leq C\norm{e}^2,
    \end{align*}
where $C = C(\norm{Du_{ex}}_\infty)$ and we have used the fact that
\begin{align*}
    u^2 - u_{ex}^2 = (e+2u_{ex})e.
\end{align*}
We have observed that $\langle D_-DD_+e_{n+1/2}, e_{n+1/2}\rangle=0$ and following estimate holds
\begin{align*}
    T_2 = -\langle\mathcal G^n , e_{n+1/2}\rangle \leq \norm{\mathcal G^n}\norm{e_{n+1/2}}.
\end{align*}
As a consequence, \eqref{temp:error_esti_1} reduces to 
    \begin{align*}
        \norm{e_{n+1}}^2 - \norm{e_{n}}^2 &\leq C\Delta t\norm{e_{n+1/2}}^2 + C \Delta t \norm{\mathcal G^n}\norm{e_{n+1/2}}\\
        &\leq C\Delta t \left(\norm{e_{n+1}}^2 +\norm{e_{n}}^2\right)+ C \Delta t \norm{\mathcal G^n}^2.
    \end{align*}
This further implies, for small $\Delta t$,
\begin{equation*}
    \norm{e_{n+1}}^2 \leq (1+C \Delta t)\norm{e_{n}}^2 + C\Delta t(\Delta t^2 + \Delta x^2)^2,
\end{equation*}
where we have taken into account \eqref{temp:error_esti}.
Since $e^0 = 0$,  we have the following estimate for $t_n\leq T$, 
\begin{equation*}
    \norm{e_{n+1}}^2 \leq e^{CT}\norm{e_0}^2 +Cn\Delta t(\Delta t^2 + \Delta x^2)^2 \leq CT (\Delta t^2 + \Delta x^2)^2,
\end{equation*}
where the constant $C$ may depend on $u$ but independent of $\Delta x$ and $\Delta t$.
Hence the result follows.   
\end{proof}
\section{Numerical Experiments}\label{sec5}
In our analysis, we provide a series of numerical illustrations of the fully discrete scheme \eqref{CNFDscheme} associated with \eqref{kdveqn}. The conventional approaches typically involve applying a numerical scheme to the periodic version of the problem, considering a sufficiently large domain where the reference solutions tend to zero outside of it, for instance, kindly refer to \cite{holden2015convergence,dutta2016convergence}.
However, in particular, our theoretical study in this paper focuses on the convergence of the approximated solution on the real line. To address this, we discretize the domain that is large enough in space for the reference solutions (exact or higher-grid solutions) to be nearly zero outside of it. Exact solutions are available for some cases, facilitating a rigorous assessment. Additionally, we evaluate our scheme's performance when dealing with initial data lacking smoothness and cases where the exact solution is unknown. In such instances, we employ a reference solution obtained with a significantly higher number of grid points. We validate the presented theoretical results and obtain better convergence rates compared to \cite{holden2015convergence}.

We introduce the relative error as
 \begin{equation*}
     E:= \frac{\|u_{\Delta x}-u\|_{L^2}}{\|u\|_{L^2}},
 \end{equation*}
 where the $L^2$-norms were computed with the trapezoidal rule in the points $x_j$ of the finest grid under consideration.
Hereby we examine the first three specific quantities, namely, \textit{mass}, \textit{momentum}, and \textit{energy} as introduced in \cite{kenig1991well}. These quantities, subject to normalization, are expressed as follows:
\begin{align*}
     C^{\Delta}_1 &:= \frac{\int_{\mathbb{R}} u_{\Delta x}\,dx}{\int_{\mathbb{R}} u_0\,dx},\qquad
     C^{\Delta}_2 := \frac{\|u_{\Delta x}\|_{L^2(\R)}}{\norm{u_0}_{L^2(\R)}},\\
     C^{\Delta}_3:&= \frac{\int_{\R} \left((\partial_x u_{\Delta x})^2 - \frac{(u_{\Delta x})^3}{3}\right)~dx}{\int_{\R} \left((\partial_x u_0)^2 - \frac{(u_0)^3}{3}\right)~dx}.
\end{align*}
Our objective is to preserve these quantities in our discrete setup. It is noteworthy that within the domain of completely integrable partial differential equations, maintaining a greater number of conserved quantities through numerical methodologies generally results in more accurate approximations compared to those preserving fewer quantities.
Furthermore, we analyze the convergence rates of the numerical scheme \eqref{CNFDscheme}, denoted as $R_E$, with varying numbers of elements $N_1$ and $N_2$. This is represented by the expression
\begin{equation*}
     \frac{\ln(E(N_1))-\ln(E(N_2))}{\ln(N_2)-\ln(N_1)},
\end{equation*}
where $E$ is considered as a function of the number of elements $N$.
\subsection{A one-soliton solution}
The family of exact solutions (one soliton) is given by
\begin{equation}\label{onesol}
    v(x,t) = 9\left(1-\tanh^2\left(\sqrt{3}/2(x-3t)\right)\right).
\end{equation}
This solution represents a single `bump' propagating to the right with a velocity of $3$. Our numerical scheme has been tested using the initial data \(u_0 = v(x,-1)\), where the solution at \(t=2\) is denoted by \(v(x,1)\). The solution of \eqref{kdveqn} is computed on a uniform grid with \(\Delta x = 20/N\) over the interval \([-10,10]\). Figure \ref{fig:KDVone} illustrates the convergence of the approximated solution. The Table \ref{tab:TableoneKdV} presents the corresponding error analysis. It is observed that the relative errors are converging to zero at the expected rates.
\begin{table}
    \centering
    \begin{tabular}{||c|c|c|c|c|c||}
        \hline
 N & E & $C^\Delta_1$ & $C^\Delta_2$ & $C^\Delta_3$ & $R_E$ \\
 \hline
 \hline
 2000 & 1.998 & 1.00 & 1.01&1.00& \\
   &  & &&& 1.10 \\
 4000  & 0.931 & 1.00 & 1.00&1.00& \\
  &  & && &1.31\\
 8000 &  0.377 & 1.00 & 1.00& 1.00&\\
   &  & &&& 1.95 \\
 16000  & 0.097   & 1.00 & 1.00&1.00&\\
  &  & & &&1.956\\
 32000  & 0.025  & 1.00 & 1.00& 1.00& \\
\hline
    \end{tabular}
   \caption{\label{tab:TableoneKdV}Relative error for one soliton solutions.}
\end{table}

\begin{figure}
    \centering
    \includegraphics[width=0.8\linewidth, height=8cm]{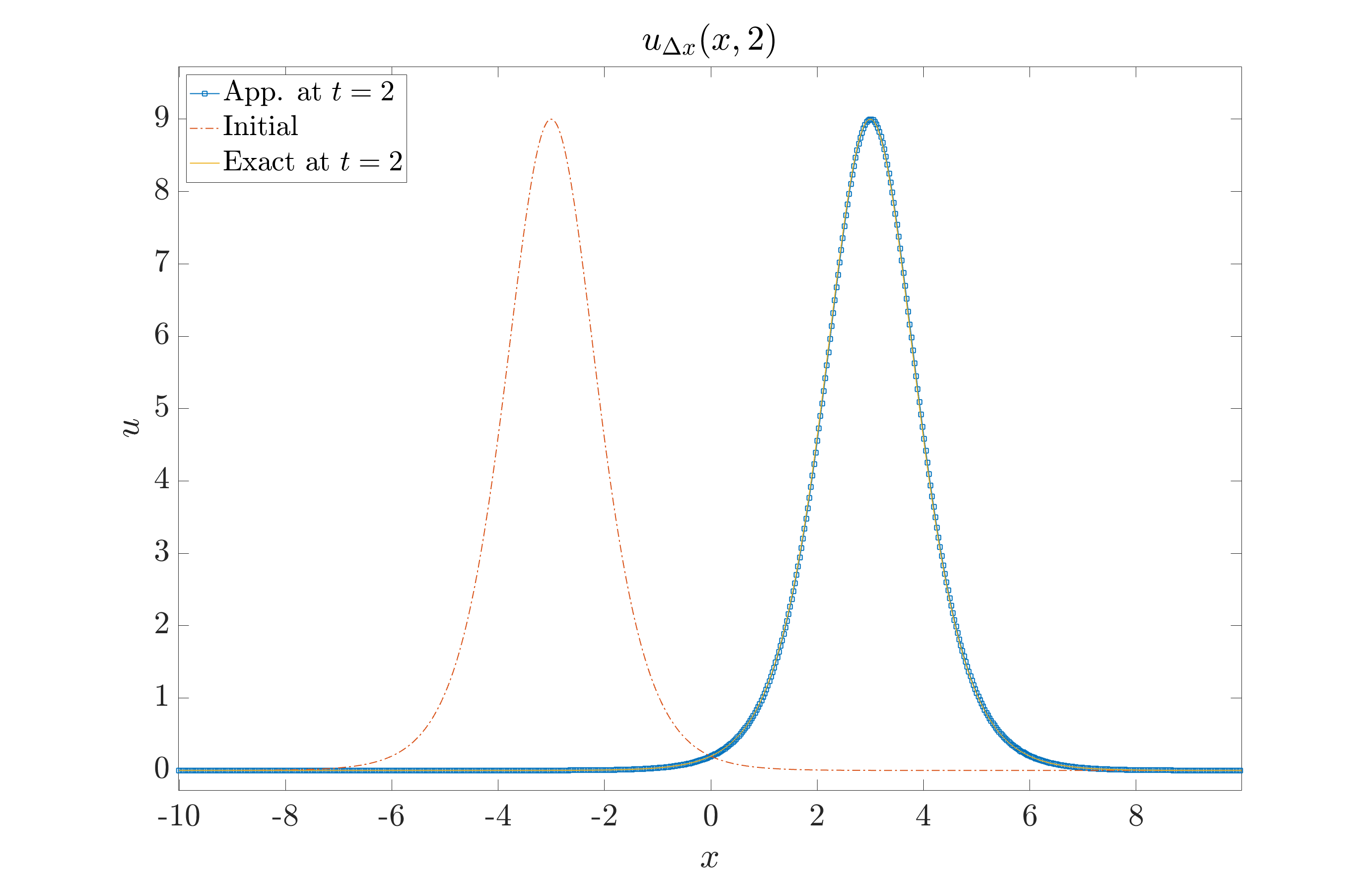}
    \caption{The exact and numerical soultion at $t=2$ with the initial data $v(x,-1)$ with $N=2000$.}
    \label{fig:KDVone}
\end{figure}
\subsection{Two soliton solution}
From a physical viewpoint, solitons with different shapes demonstrate different speeds, establishing a connection between the height and speed of a soliton. A taller soliton moves faster than a shorter one. When two solitons move across a surface, the taller soliton overtakes the shorter one, and both solitons remain unchanged after the collision. This situation introduces a significantly more complex computational challenge than solving for a solitary soliton solution.

In the case of two soliton, the family of exact solutions (see \cite{dutta2015convergence}) of the KdV equation is given by 
 \begin{equation}\label{twosol}
     w(x,t) = 6(b-a) \frac{b \csch^2\left(\sqrt{b/2}(x-2bt)\right)+a \sech^2\left(\sqrt{a/2}(x-2at)\right)}{\left(\sqrt{a}\tanh\left(\sqrt{a/2}(x-2at)\right) - \sqrt{b} \coth\left(\sqrt{b/2}(x-2bt)\right)\right)^2}.
\end{equation}
for some constants $a$ and $b$. We have considered the parameters $a=0.5$ and $b=1$, and the initial data $u_0(x) = w(x,-10)$.
We computed the approximated solution at time $t = 20$ to compare with the exact solution $w(x, 10)$. The Figure \ref{fig:KdVtwoCN} represents the exact solution at $t=-10$, $t=0$ and $t=10$ along with numerical solution at $t=0$, $t=10$ and $t = 20$. Since the wave is relatively narrow, the $L^2$-error assumes significant proportions. Table \ref{tab:TableKDVCN} demonstrates the relative $L^2$-errors for the two soliton simulations.

\begin{table}
    \centering
    \begin{tabular}{||c|c|c|c|c|c||}
        \hline
 N & E & $C^\Delta_1$ & $C^\Delta_2$ & $C_3^\Delta$&$R_E$ \\
 \hline
 \hline
  500  & 6.010 & 1.01 & 1.00 & 1.00 &\\
  &  &  &  & & 1.70\\
 1000  & 1.848 & 1.00 & 1.00 &1.00 & \\
  &  &  &  && 1.92\\
 2000 & 0.488 & 1.00 & 1.00 &1.00 &\\
   &  &  &  & &1.93 \\
 4000 & 0.128 & 1.00 & 1.00 &1.00 &\\
  &  &  &  & &2.04 \\
  8000 & 0.031 & 1.00 & 1.00 &1.00 &\\
  &  &  &  & & \\
\hline
    \end{tabular}
   \caption{The relative errors between exact and numerical solution obtained at $t=20$ with the initial data $u_0(x) = w(x,-10)$.}
   \label{tab:TableKDVCN}
\end{table}

\begin{figure}
    \centering
    \includegraphics[width=1 \linewidth, height=9cm]{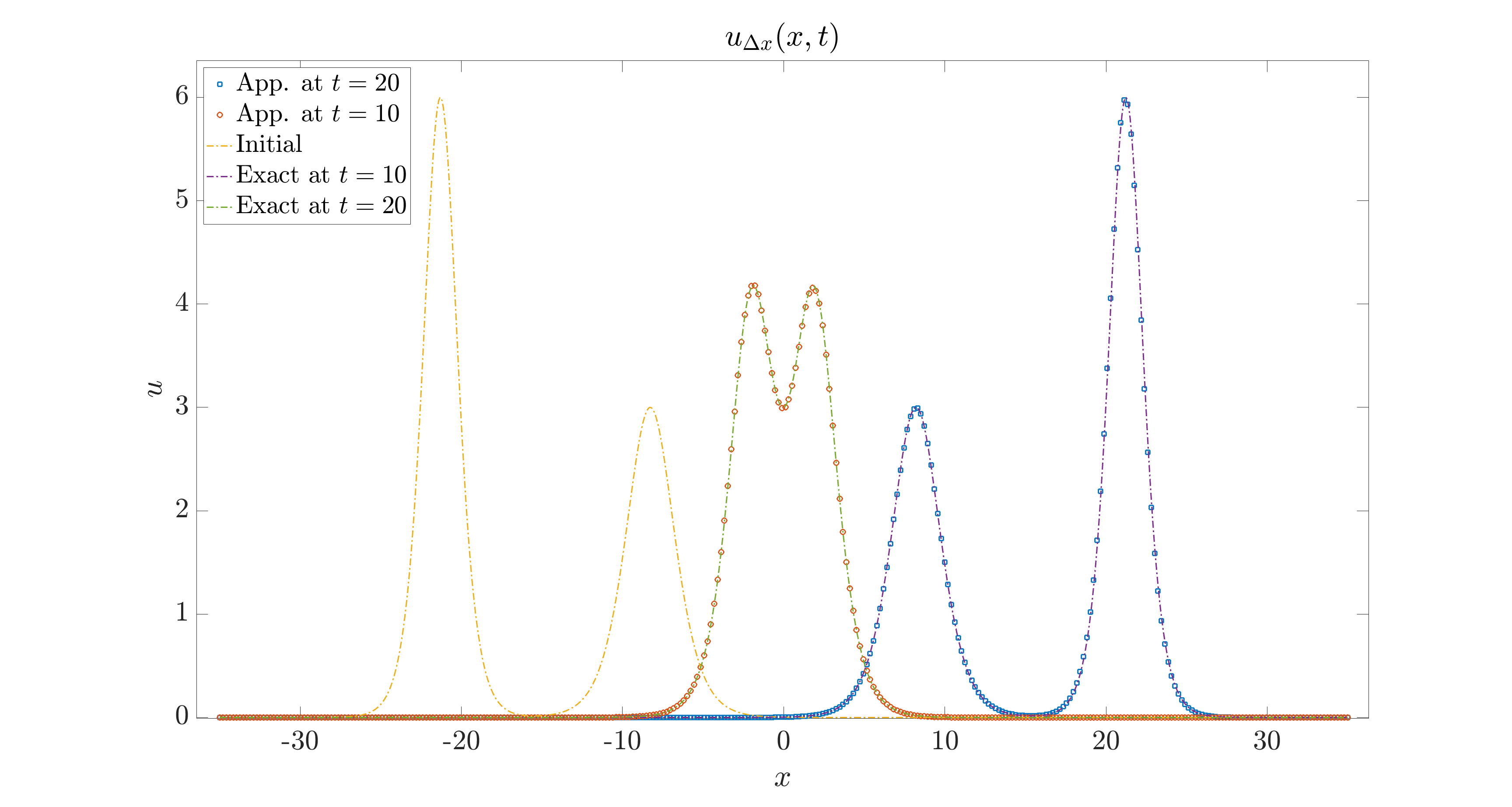}
    \caption{The exact and numerical solution obtained at $t=40$ with the initial data $u(x,-20)$ with $N=1000$.}
    \label{fig:KdVtwoCN}
\end{figure}
   \subsection{Non smooth Initial data} 
Our numerical experiments included a case where the initial data resides in $L^2$ but does not belong to any Sobolev space with a positive index. Importantly, the conclusions drawn in Section \ref{sec4} remain applicable when considering the periodic scenario. To illustrate, we specifically consider the following initial data:
\begin{equation*}
    u_0(x) =
    \begin{cases}
        \frac{1}{2} (x+1), & \text{for } x \in [-1,1], \\
        0, & \text{otherwise}.
    \end{cases}
\end{equation*}
This initial data is defined in the interval $[-5,5]$ and it is periodically extended beyond this interval.
It is crucial to note that in this specific case, an exact solution is not known. In Table \ref{tab:Tablealpha12L2}, we show the $L^2$-errors using the approximate solution with $N=32000$ grid points as a reference solution at time $T = 0.5$.
The notable errors and slow convergence rate suggest that we have not reached the asymptotic zone yet. Increasing grid points was impractical and did not yield better results compared to our current reference solutions as references may not be close to the exact solution.

In Figure \ref{fig:L2CN}, we have plotted the approximate solution using $N=16,000$ grid points against the reference solution. We have observed that the reference solution contains many high frequency waves which also have a very high speed. Our devised method captures the main features but struggles to resolve the finer details of the solution.
\begin{table}
    \centering
    \begin{tabular}{||c|c|c|c|c|c||}
        \hline
 N & E & $C^\Delta_1$ & $C^\Delta_2$ & $C_3^\Delta$ & $R_E$ \\
 \hline
 \hline
 250  & 0.4730  & 0.016 & 0.12& 0.597& \\
  &  &  &  & & -0.032\\
 500 & 0.4836 & 0.032& 0.18 & 0.551&\\
   &  &  &  & &0.186\\
 1000 & 0.4250 & 0.062& 0.25  & 0.574& \\
  &  &  &  && -0.050 \\
  2000 & 0.4399  & 0.127& 0.35 & 0.553 &\\
  &  &  & && 0.119\\
  4000 & 0.4050 & 0.254  & 0.50 & 0.530 &\\
  &  &  &  & & 0.195\\
  8000 & 0.3539 & 0.496 & 0.71  & 0.552 &\\
  &  &  &  & &\\
\hline
    \end{tabular}
   \caption{ Relative errors with $L^2$ initial data.}
   \label{tab:Tablealpha12L2}
\end{table}
\begin{figure}
    \centering
    \includegraphics[width=1 \linewidth, height=9cm]{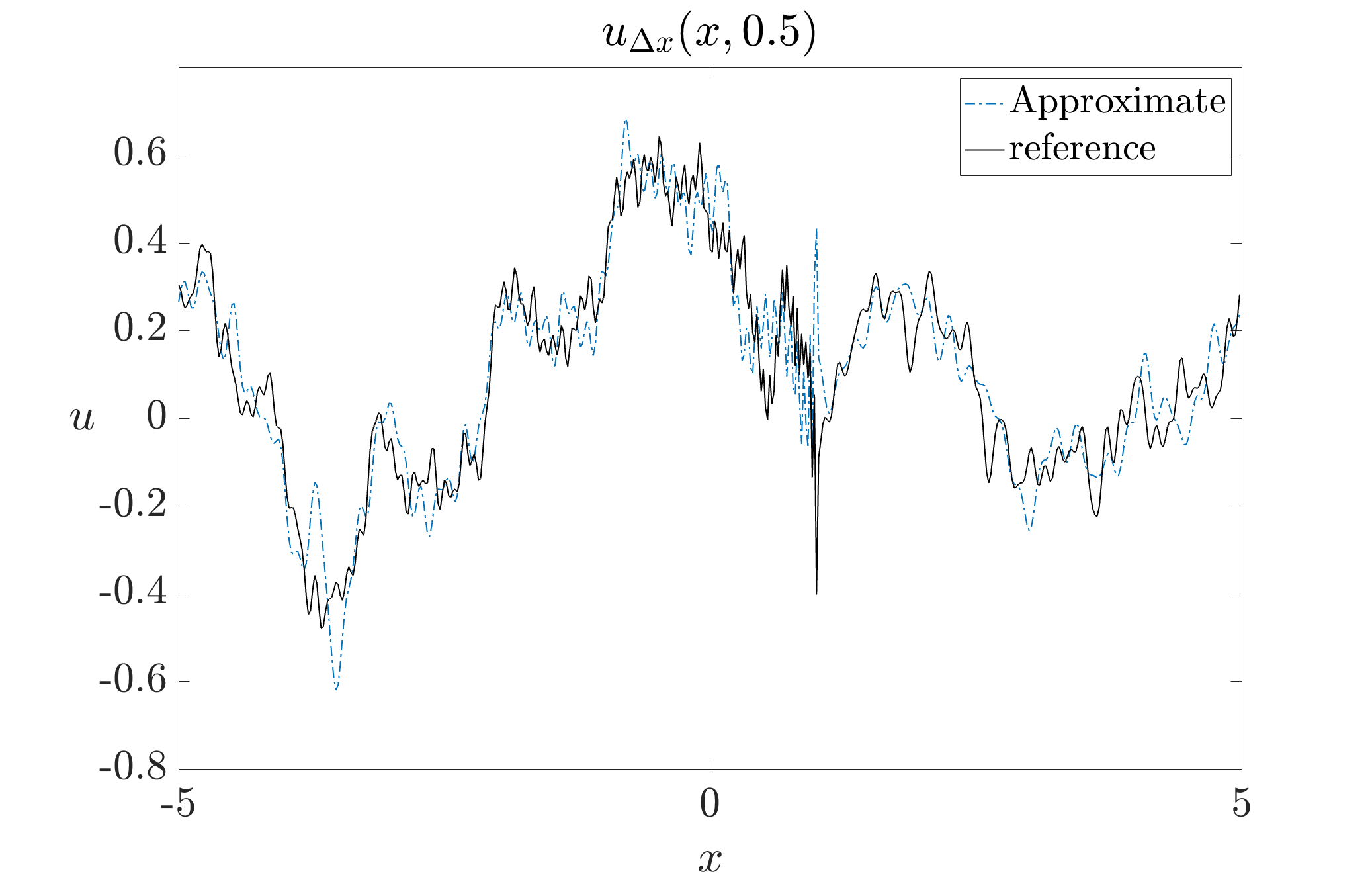}
    \caption{The reference (with $N=32000$) and numerical solution at $t=0.5$ with the $L^2$ initial data with $N=16000$.}
    \label{fig:L2CN}
\end{figure}

In conclusion, our study addresses the computational challenges inherent in solving the KdV equation and develops an efficient numerical scheme. The proposed conservative scheme performs reasonably well in practice and has proven to converge. The inherent smoothing effect plays an important role in the convergence analysis.
The second order  convergence signifies a substantial advancement in accurate and efficient computing solutions of KdV equation.


\section*{Acknowledgements}
The authors would like to thank Rajib Dutta and Ujjwal Koley for several fruitful discussions.

\end{document}